\documentclass{article}

\usepackage[utf8]{inputenc}

\usepackage{mathrsfs}
\usepackage{amsmath}
\usepackage{hyperref}
\usepackage{amsthm}
\usepackage{amssymb}
\usepackage{parskip}
\usepackage{subcaption}
\usepackage{graphicx,caption}
\usepackage{xcolor}
\usepackage{enumitem}
\usepackage{lipsum}

\theoremstyle{plain}
\newtheorem{theorem}{Theorem}[section]
\newtheorem{corollary}[theorem]{Corollary}
 
\newtheorem{prop}[theorem]{Proposition}
\newtheorem{proposition}[theorem]{Proposition}
\newtheorem{lemma}[theorem]{Lemma}

\theoremstyle{definition}
\newtheorem{definition}[theorem]{Definition}
\newtheorem{remark}[theorem]{Remark}

\def \R{\mathbb R}
\def \N{{\mathbb N}}
\def\E{\mathbb E}
\def\P{\mathbb P}
\def \Z{\mathbb Z}
\newcommand{\cA}{\mathcal{A}}
\newcommand{\cB}{\mathcal{B}}
\newcommand{\cC}{\mathcal{C}}

\newcommand{\cE}{\mathcal{E}}
\newcommand{\cF}{\mathcal{F}}
\newcommand{\cG}{\mathcal{G}}

\newcommand{\cI}{\mathcal{I}}
\newcommand{\cJ}{\mathcal{J}}
\newcommand{\cK}{\mathcal{K}}

\newcommand{\cM}{\mathcal{M}}

\newcommand{\cO}{\mathcal{O}}
\newcommand{\cQ}{\mathcal{Q}}

\newcommand{\cX}{\mathcal{X}}
\newcommand{\cY}{\mathcal{Y}}
\newcommand{\cZ}{\mathcal{Z}}

\newcommand{\bigslant}[2]{{\raisebox{.2em}{$#1$}\left/\raisebox{-.2em}{$#2$}\right.}}

\hypersetup{
	colorlinks=true,
	linkcolor=blue,
	citecolor=blue,
	filecolor=magenta,      
	urlcolor=cyan,
	linktoc=page
}

\newcommand{\floor}[1]{{\left\lfloor #1 \right\rfloor}}
\newcommand{\ceil}[1]{{\left\lceil #1 \right\rceil}}

\newcommand{\fW}{\mathfrak{W}}

\newcommand{\sB}{\mathscr{B}}

\newcommand{\sD}{\mathscr{D}}

\newcommand{\sP}{\mathscr{P}}

\newcommand{\sW}{\mathscr{W}}
\newcommand{\oa}{{\overline{a}}}
\newcommand{\del}{\partial}

\newcommand{\sset}{\subset}
\newcommand{\lf}{\left}
\newcommand{\rg}{\right}

\newcommand{\ga}{\gamma}

\newcommand{\ep}{\epsilon}

\newcommand{\ka}{\kappa}
\newcommand{\de}{\delta}
\newcommand{\be}{\beta}

\newcommand{\sig}{\sigma}

\newcommand{\la}{\lambda}
\newcommand{\al}{\alpha}

\newcommand{\smin}{\setminus}

\newcommand{\fR}{\mathfrak{R}}
\newcommand{\fB}{\mathfrak{B}}

\newcommand{\fC}{\mathfrak{C}}
\newcommand{\fD}{\mathfrak{D}}
\newcommand{\fF}{\mathfrak{F}}
\title{Spread of infections in a heterogeneous moving population}
\author{Duncan Dauvergne and Allan Sly}

\begin{document}

\maketitle

\begin{abstract}
We consider a model where an infection moves through a collection of particles performing independent random walks. In this model, Kesten and Sidoravicius established linear growth of the infected region when infected and susceptible particles move at the same speed. In this paper we establish a linear growth rate when infected and susceptible particles move at different speeds, answering an open problem from their work. Our proof combines an intricate coupling of Poisson processes with a streamlined version of a percolation model of Sidoravicius and Stauffer.
\end{abstract}


	\section{Introduction}
	
We consider an interacting particle system with two types of random walkers, labelled ${\bf S}$ and ${\bf I}$ for susceptible and infected (in other papers these are often called $A/B$ or sometimes $X/Y$).  When a susceptible particle meets an infected one, it too becomes infected. Special cases of this model include multi-particle Diffusion Limited Aggregation (MDLA) and the Frog Model.

Taking terminology from the epidemiology literature, we refer to this interacting particle system as an SI process. Formally, we define the SI process $X_t$ as follows. The process is initialized with a sea of particles distributed according to a rate $\mu$ Poisson process on $\Z^d$.  All particles are initially {\bf susceptible} and each susceptible particle performs a continuous-time rate-$D_S$ simple random walk on $\Z^d$. At time $0$, a single {\bf infected} particle is planted at the origin, which performs a continuous-time rate-$D_I$ simple random walk. The particles interact as follows: when a susceptible particle is at the same vertex as an infected particle, it immediately becomes infected, at which point it starts performing a  rate-$D_I$ random walk.  We let ${\bf S_t}$ and ${\bf I_t}$ denote the set of susceptible and infected  particles at time $t$, respectively.  For a particle $a$ we will let $\oa(t)$ denote its position at time $t$.  
Our main theorem states that for large $t$, every particle in a ball of radius $\ga t$ is infected.

\begin{theorem}\label{thm:main}
For any $\mu,D_I,D_S>0$ and any dimension $d\geq 1$ there exists $\gamma(\mu,D_I,D_S,d)>0$ and a random time~$T$ such that $T<\infty$ almost surely and for all $t>T$,
\[
\inf_{a\in {\bf S_t}} |\oa(t)|  \geq \gamma t.
\]
\end{theorem}

Kesten and Sidoravicius~\cite{kesten2005spread} established Theorem~\ref{thm:main} in the special case when $D_I = D_S$.  Their proof made use of the special property that when the speeds of the two types are equal, the infection process does not affect the trajectories of the particles. In particular, the particles viewed without labels perform independent simple random walks and remain in equilibrium. Because of this, it was possible for them to study a Poisson process on the space of paths. They performed a multiscale analysis to show that regions with atypically few particles are rare and sparsely arranged, so the infection process can spread quickly. In the case $D_I \neq D_S$ this analysis no longer works, as the infection process affects the trajectories and the two cannot be studied separately.  In this setting, Kesten and Sidoravicius stated they could show that all partices in a ball of radius $t\log^{-C}(t)$ are infected by time $t$, while omitting a proof. They left the case of unequal speeds as an open problem. Adding to the challenge, when $D_I \neq D_S$ the infection process is not monotone in any of the parameters of the model $\mu,D_I,D_S$. The process is not even monotone in the addition of new particles to the initial condition, as an additional particle may change the time of one infection, possibly delaying a later infection.

We note that the linear speed in Theorem \ref{thm:main} is the correct growth rate. Indeed,~\cite{kesten2005spread} showed the following result.

\begin{theorem}[\cite{kesten2005spread}, Theorem 1]\label{thm:KS.upper.bound}
For any $\mu,D_I,D_S>0$ and $d \in \N$ there exists $\gamma'(\mu,D_I,D_S,d)>0$ such that for large enough $t$
\[
\P\Big[\sup_{a\in {\bf I_t}} |\oa(t)|  \geq \gamma' t \Big] \leq 2e^{-t}.
\]
\end{theorem}

\begin{remark}
	\label{R:variations}
	Our framework for proving Theorem \ref{thm:main} is quite robust and can accommodate variations of the setup. For example, the same proof also works if we modify the infection rule by only having susceptible particles become infected with probability $p \in (0, 1)$ each time they meet an infected particle, or if we allow infections to occur when susceptible particles enter a given radius of an infected particle. The proof also goes through verbatim if we start from any initial condition that stochastically dominates a Poisson initial condition, even though the SI process is not monotone in the initial condition. 
\end{remark}

\subsection{Proof Sketch}

The basic premise of the proof is that if the infection process is spreading linearly then it will keep visiting new regions of the plane with independently distributed particle configurations.  It infects the particles in each new region, some of which spread out which continues the linear spread of the process. Of course, this is only approximately true and sometimes a new area will by chance only be sparsely populated with susceptible particles.  Our proof is composed of two parts: an analysis of a competing growth model known as Sidoravicius--Stauffer percolation that will allow us to deal with sparsely populated regions, paired with a coupling of Poisson processes that can be thought of as a way to make rigorous the notion of constantly discovering new independent particles.

In order to preserve independence, we reveal the process only in a limited but growing portion of the plane.  Partitioning the plane into blocks, we ``colour'' a block the first time an infected particle enters that block (red) or if a certain time has elapsed since a neighbouring block was coloured (blue); the blue colouring ensures that the colouring process grows with a linear speed.  We only reveal particles when they are first in a coloured block.  For each block we generate an independent Poisson process on the space of particle trajectories. When a new block is coloured we use this Poisson process to decide what particles are revealed. To account for the dependence in the original SI process, we only add particles whose past trajectories would not have previously been in a coloured region. This is the correct conditional distribution. Since our colouring process is moving at linear speed whereas the random walks are diffusive, the probability that a particle trajectory was previously in a coloured region does not increase with $t$, ensuring that we typically reveal a large number of particles in each block.

Using this construction we can ``lower bound'' the number of particles and the spread of the infection in each newly discovered coloured block.  When a block is coloured red, an infected particle has entered and there will typically be a large number of newly discovered particles that it can infect in a short period of time. At least some of these particles will likely move to a neighbouring block rather quickly, propagating the infection. To make this precise, we define events $\cA_B$ for each block which have high probability, are almost independent (formally they are 1-dependent), and guarantee that if a block is red then all its neighbouring blocks will contain an infected particle after a short period of time. We have good control over $\P [\cA_B]$. In our setup, we can guarantee that the probability of the events $\cA_B$ will increase to $1$ as we use coarser blocks and slower growth rates for the blue process.

In order to show that most blocks are red and that the infection process is spreading quickly, we will relate the colouring process to {\bf Sidoravicius--Stauffer Percolation} (SSP), introduced in ~\cite{sidoravicius2019multi} \footnote{It was introduced with the somewhat unwieldy name First Passage Percolation in a Hostile Environment (FPPHE)}. This process, which was first used to study MDLA, is given by the competition between two i.i.d.\ first-passage percolation processes. The first of these is red, and is initiated at the origin at time $0$. As the red process tries to spread, it may encounter {\bf blue seeds} at certain vertices. Whenever it attempts to invade a blue seed, the blue growth process becomes activated from that vertex and starts to spread, competing with the red process to annex vertices. 
Sidoravicius and Stauffer used a multiscale construction to show that if the red growth rate is faster than the blue growth rate, $d \ge 2$, and the blue seeds are given by site percolation with small enough probability $p$, then with positive probability, most of $\Z^d$ becomes red.

%
%
%
%
%

We construct a simplified version of SSP in which the blue process spreads along an edge $(u, v)$ after a waiting time $\kappa(u, v) \gg 1$ while the red process spreads in time at most $1$. This simplifies the multi-scale analysis as we need only consider the random locations of the blue seeds and can ignore the randomness in the passage times. This is important, since in our setting these passage times may have complicated dependencies. This simplification also allows us to have a caveat where the red process can potentially `tunnel' through the blue process, a phenomenon that arises in our coupling with the SI model. 

 We take the blue seeds to be the blocks where $\cA_B^c$  holds and show that the blue blocks in the SI colouring process can be stochastically dominated by the blue set of our SSP.  While the red component of the SSP is infinite only with positive probability, we can make this probability go to $1$ by taking blocks with larger side-lengths. This, however, comes at the expense of a slower overall growth rate, and hence only shows Theorem \ref{thm:main} with a \emph{random} $\ga>0$. To get around this, we consider how SSPs with smaller side-lengths and different values of $\ka$ are coupled to the same SI process. Analyzing the relationships between the spread of these processes allows us to conclude Theorem \ref{thm:main} with a \emph{deterministic} $\ga > 0$ for all dimensions $d \ge 2$.

The tools of Sidoravicius-Stauffer percolation are no longer relevant in dimension~$1$, where the red process can never go around a blue seed. However, in this setting we can use a simpler proof idea, still based on a coupling of the SI process with different Poisson processes, where we look at block sizes of varying scales. This Poisson framework for the proof is related to the framework used in the study of supercritical MDLA in dimension $1$ in \cite{sly2016one}, though the details of the proofs are different.

\subsection{Prior work}

Two special cases of this model have received considerable attention.  The Frog Model, also known as the Stochastic Combustion process, is the case $D_A=0$ where susceptible particles do not move until they become infected.  This substantially simplifies the analysis since a particle is always infected at its initial location.  Consequently a great deal is know for the Frog Model including a shape theorem and large deviations~\cite{alves2002shape,ramirez2004asymptotic,comets2009fluctuations,berard2010large}.  

The other natural special case is $D_B=0$, also known as multi-particle Diffusion Limited Aggregation (MDLA).  Here the infection rule must be modified slightly since otherwise all infected particles would be frozen on the same vertex.  Instead, particles become infected just before they move onto an infected vertex and remain frozen in place there, along with any other susceptible vertices at the same location.  This is similar to the well known Diffusion Limited Aggregation model where particles enter one at a time from infinity and perform a random walk until they hit the aggregate. Instead, in MDLA, there is a sea of susceptible particles which move as random walks until they hit the aggregate and freeze.  Sidoravicius and Stauffer~\cite{sidoravicius2019multi} showed that in a variant where particles perform exclusion processes that there is linear growth with positive probability if the density is close enough to $1$ in dimension $2$ and higher. In the random walk version of the model, linear growth with probability $1$ was established in~\cite{sly2016one} whenever the density is at least $\frac56$.  The question of a shape theorem remains open.
An unpublished conjecture of Eldan predicts that there is no phase transition. That is, there is linear growth for all positive initial densities in dimension 2 and higher.  

In fact, a phase transition in MDLA is expected to occur only in dimension~1.  Kesten and Sidoravicius~\cite{kesten2008problem} showed that the diameter of the aggregate grows like order $t^{\frac12}$ when the initial density is less than 1 while~\cite{sly2016one} showed that the aggregate grows linearly when the initial density is greater than 1.  In the critical case Elboim, Nam and the second author~\cite{elboim2020critical} showed that the speed is $t^{\frac23}$ and the scaling limit is given by the integral of a Bessel$(8/3)$ process to the power of $-\frac23$.  A discontinuous scaling limit was found in a simplified \emph{frictionless} version of the model by Dembo and Tsai~\cite{dembo2019criticality}.

A natural question in any growth model is to establish a limiting shape theorem.  In the case of a simple model like first passage percolation this is established by a subadditivity argument~\cite{richardson1973random,cox1981some}.  This does not hold in the SI model. However, in the case $D_S = D_I$, Kesten and Sidoravicius~\cite{kesten2008shape} were able to establish an approximate superconvolutivity property, roughly showing that superadditivity happens on average. In doing so they established a shape theorem.  Their approach heavily uses monotonicity results that are not available when $D_S \neq D_I$, although we still expect them to be approximately true.

Our approach can be viewed as a synthesis of two main ideas.  The first is that revealing the particles and their trajectories corresponds to revealing part of a Poisson process on the space of random walk paths and that the information revealed is a stopping set which leaves the remainder of the Poisson process independent.  A simpler version of this approach was used to study MDLA~\cite{sly2016one,elboim2020critical}.  With this perspective, we are able to achieve a certain amount of spatial independence which plays a crucial role in our proof. 

The second main idea is to couple the block process with Sidoravicius--Stauffer Percolation for competing growth models.  
The application of SSP in~\cite{sidoravicius2019multi} required the case where the competing processes grow with almost equal speed (see also~\cite{finn2020non}). Our framework requires less from SSP. In particular, we can allow a large deterministic separation of the speeds of the competing growth processes. While our conditions are different from~\cite{sidoravicius2019multi} and so we reprove those results, the proofs in that part of the paper are essentially simplified  versions of theirs. In particular, the fact that the blue process grows deterministically means that we need only consider the randomness of the seed locations and not the passage times which substantially streamlines the argument.  

The combination of the two approaches also lends itself to proving linear growth with probability 1. Our local analysis of the infection process holds with higher probability as we take the block size to infinity. This, however, comes at the expense of a slower overall growth rate. By coupling to SSPs at different scales, we show how the SI process can get a good start with larger blocks and then move down to smaller blocks and a faster growth rate.

\subsection{Future work}

In a forthcoming paper, we analyze a version of the model where particles recover from the infection at some rate $\la$, after which point they become immune. The framework and technical tools introduced in this paper are crucial for understanding that more complicated model.

\subsection*{Acknowledgments}
The authors would like to thank Dimitris Cheliotis and Alexander Stauffer for useful discussions.  AS was supported by NSF grants DMS-1855527 and DMS-1749103, a Simons Investigator grant and a MacArthur Fellowship.

\subsection{Outline of the paper}

In Section \ref{S:SSP}, we develop the theory of Sidoravicius--Stauffer percolation in our setting. In Section~\ref{S:linear} we use this theory along with a Poisson block construction of the SI process and simple estimates about random walks to prove Theorem~\ref{thm:main} in dimensions $d \ge 2$. Two technical but intuitive points in Section~\ref{S:linear} are left for later sections: a rigorous justification of the construction of the SI process in terms of independent Poisson processes (addressed in Section~\ref{s:coupling}) and a check that the colouring process we define for the SI model yields an SSP (Section \ref{S:SI-to-BR}). Section~\ref{S:dim-1} proves Theorem \ref{thm:main} in the remaining simpler case of dimension one.

All arguments in the paper in the case when $d \ge 2$ are identical in all dimensions. Because of this, to simplify the notation and exposition we fix the dimension equal to $2$ in Sections \ref{S:SSP},  \ref{S:linear}, \ref{s:coupling}, and \ref{S:SI-to-BR}.
	
\section{Sidoravicius--Stauffer percolation}
	\label{S:SSP}

	In this section we develop a deterministic theory of competing red and blue growth processes known as Sidoravicius--Stauffer percolation (SSP). Informally, in SSP, a set of red vertices grows outwards from the origin in $\Z^2$ according to a set of edge weights, similarly to first passage percolation. Within $\Z^2$, there are a collection of blue seeds $\mathfrak{B}_*$ which cannot be invaded by the red process. Whenever the red process attempts to invade a blue seed, the blue proces becomes activated. Once activated, the blue process will also grow outward from a blue seed at a slow, constant speed. The red and blue processes can only invade squares that are not already part of one of the two coloured processes. 
	
	Within this framework, we also allow for squares adjacent to the blue process to become activated at arbitrary times. When these squares are activated, they will turn red and join the red process unless they belong to the set of blue seeds, in which case they will turn blue.
	
	\subsection{Constructing the process}
	
	We will define a pair of coupled growth processes $\fR, \fB:[0, \infty) \to \{S: S \sset \Z^2\}$. The growth of these processes will be governed according to weights on the set of directed edges
	$$
	E := \{(u, v) \in \Z^2 : \|u - v\|_1
	= 1 \}.
	$$ 
	We will need the following data to define $\fR, \fB$.
	\begin{itemize}
		\item A function $X_\fR:E \to [0, 1]$. We will think of $X_\fR$ as a collection of clocks which govern the growth of the red process.
		\item A collection of blue seeds $\fB_* \sset \Z^2$.
		\item Functions $X_\fB:E \to [0, \infty)$ and $\kappa:E \to (0, \infty)$ such that $X_\fB \le \ka$. We will think of $X_\fB$ as a collection of clocks which govern the growth of the blue process, and we call $\ka$ the \textbf{parameter} of the process. 
		
		Edges where  $X_\fB(e) = \ka(e)$ will encourage the spread of the blue process, whereas other edges will encourage the spread of the red process. When thinking about the proofs in this section, it may be helpful for the reader to think of $\ka$ as constant. There will be no extra nuance for the case of variable $\ka$.
		\item We have allowed for the caveat that clocks may have a value of $0$. This corresponds to instantaneous invasion. To ensure that the process is still well-defined even when instantaneous invasion is allowed, we require that there are no directed cycles $(u_0, u_1, \dots u_n = u_0)$ with
		$$
		X_\fR(u_{i-1}, u_{i}) \wedge X_\fB(u_{i-1}, u_{i}) = 0
		$$
		for all $i = 1, \dots, n$. 
	\end{itemize}
	We define $\fR$ and $\fB$ by recording a time $T(u)$ when each vertex $u$ gets added to the process, and a colour $C(u)$ for that vertex. We initialize the process by setting $T(0) = 0$, and setting $C(0) = \fR$ if $0 \notin B$ and $C(0) = \fB$ if $0 \in \fB_*$. The rules for assigning the other times and colours are as follows.

	For every edge $(u, v) \in E$, at time $T(u) + X_{C(u)}(u, v)$, the edge $(u, v)$ will ring, and the process will update accordingly. If $T(v) < T(u) + X_{C(u)}(u, v)$, then nothing happens. This corresponds to the case when $v$ has already been added to one of the processes by this time. Otherwise, we set $T(v) = T(u) + X_{C(u)}(u, v)$, and colour $v$ according the following four rules.
	\begin{enumerate}
		\item If $v \in \fB_*$, set $C(v) = \fB$.
		\item If $C(u) = \fR$ and $v \notin \fB_*$, then set $C(v) = \fR$.
		\item If $C(u) = \fB, v \notin \fB_*$, and $X_{\fB}(u, v) < \ka(u, v)$, then set $C(v) = \fR$.
		\item If $C(u) = \fB, v \notin \fB_*$ and $X_{\fB}(u, v) = \ka(u, v)$, then set $C(v) = \fB$.
	\end{enumerate}
	One caveat with this set of rules, is if $T(u) + X_{C(u)}(u, v) = T(u') + X_{C(u')}(u', v)$ for two different vertices $u, u'$. In this case, if the colours assigned to $v$ via the two edges $(u, v)$ and $(u, v')$ are different, we set $C(u) = \fB$. For $t \in [0, \infty)$, we set
	\begin{equation}
	\label{E:BR-def}
	\begin{split}
	\fR(t) &= \{u \in \Z^2 : T(u) \le t, C(u) = \fR\} \quad \text{and} \\
	\quad \fB(t) &= \{u \in \Z^2 : T(u) \le t, C(u) = \fB\}. 
	\end{split}
	\end{equation}
	When working with SSPs, we will always assume that for any $t$, only finitely many squares $u$ satisfy $T(u)  < t$.
	Under this \textbf{finite speed} assumption, the colouring process is well-defined. From now on, all processes we consider have finite speed.
	
	The above construction defines an SSP on $\Z^2$ driven by an initial red site at the origin at time $0$. We will also consider SSPs on certain connected subsets of vertices $V \sset \Z^2$, activated by an external red source at the \textbf{boundary} 
	$$
	\del V = \{u \in V : (u, v) \in E \text{ for some } v \notin V\}.
	$$
	We assume $\del V$ is finite.
	These processes are defined in the following way. Let $S:\del V \to [0, \infty)$ be any function such that $S(x) < \infty$ for at least one $x$. We call $S$ the \textbf{source function}.
	The evolution of the SSP on $V$ with source $S$ will be governed in the same way as an SSP on $\Z^2$ via a collection of blue seeds $\fB_* \sset V \smin \del V$ and weight functions $X_\fR, X_\fB$ with domain $V^2 \cap E$, with the following changes.
	
	Let $I \sset \del V$ be the set of vertices on which $S$ is minimal, and let $t$ denote this minimal value. We initialize the process at time $t$ by setting $T(u) = t$ and $C(u) = \fR$ for all $u \in I$. We run the process from this configuration, with the additional rule that if $u \in \del V$ and $S(u) \in (t, \infty)$, then we also set $T(u) = S(u)$ and $C(u) = \fR$ if $T(u)$ has not already been assigned to a value less than $S(u)$ by that time. 
	
	An SSP on $V \sset \Z^2$ can be used to describe the restriction of a global SSP to $V$, if we know that the blue process never attempts to invade $V$ from $V^c$. 
	
	Note that these definitions also make sense if the source function $S$ is defined on an arbitrary subset of $V \smin \fB_*$, rather than just a boundary. For example, later on it will be convenient to add the origin $0$ to the set on which $S$ is defined, since this site will initialize the red process.
	
	\subsection{An encapsulation theorem}
	\label{SS:encapsulate}
	
	The main goal of Section \ref{S:SSP} is to show that under certain conditions on an SSP, if $\ka$ is bounded below by a sufficiently large constant, then we can find a set $\mathfrak{C}$ depending only on the initial configuration of blue seeds such that $\fB(\infty) \sset \mathfrak{C}$. We will do this with a multiscale argument, and the set $\mathfrak{C}$ will be built in terms of these scales. Note that the restriction on $\ka$ is purely to facilitate the proof; a version of the same result holds for any $\ka > 1$.
	
	We first set up the series of required scales. Here and throughout the paper, various explicit constants appear. These have no meaning and are not chosen optimally -- we only work with explicit constants to help orient the reader with the large number of different scales that must be balanced.  
	
	Let $r_0 = 1 \le r_1 < r_2 < \dots$ be a sequence of integer scales satisfying
	\begin{equation}
	\label{E:gammabound}
	\ga := \prod_{i=0}^\infty \left(1 + \frac{10^{12} r_i}{r_{i+1}} \right) < 2,
	\end{equation}
	and such that $r_{i+1} = (i+1)^2 r_i$ for all large enough $i$.
	Condition \eqref{E:gammabound} implies that $10^{12} r_i \le r_{i+1}$ for all $i$ and the relationship between $r_i, r_{i+1}$ implies that there exists a constant $c > 0$ such that
	\begin{equation}
	\label{E:ri-relationship}
	c^{-1} \le \frac{r_n}{(n!)^2} \le c
	\end{equation}
	for all $n$. 
	To define the sets used to form $\fC$ we first set up some notation. For closed disks and annuli, we write
	$$
	D(x, r) = \{y \in \Z^2: |y - x| \le r \}, \qquad A(x, r, R) = \{y \in \Z^2 : r \le |y - x| \le R\},
	$$
	where $|x - y|$ denotes graph distance with the edge set $E$, or equivalently the $L^1$-distance. Here and throughout the paper, for two sets $A, B$, we use the notation
	$$
	d(A, B) = \inf \{d(x, y) : x \in A, y \in B\},
	$$
	for the infimal distance between two sets, and write $d(x,A):= d(\{x\}, A)$. We will also use the union-of-disks notation
	$$
	D(A, r) = \{y \in \Z^2: d(y, A) \le r \}= \bigcup_{x \in A} D(x, r).
	$$
	Now, for a set of blue seeds $\fB_\diamond \sset \Z^2$, define
	$$
	A_1(\fB_\diamond) := \{x \in \fB_\diamond: (\fB_\diamond \setminus \{x\}) \cap D(x, r_1/3) = \emptyset\},
	$$
	We then recursively define $\fB_{\diamond, k} = \fB_\diamond \smin \cup_{i=1}^{k-1} A_i(\fB_\diamond)$, and let 
	$$
	A_k(\fB_{\diamond}) := \{x \in \fB_{\diamond, k}: \fB_{\diamond, k} \cap A(x, r_{k-1},  r_k/3) = \emptyset \}.
	$$
	Finally, for a set $A \sset \Z^2$, we let $[A]$ be the union of $A$ and all bounded components of $A^c$. Our main theorem about SSPs is the following. 
	\begin{theorem}
		\label{T:BssetI}
		Consider an SSP $(\fR, \fB)$ on $\Z^2$ initiated from a set of blue seeds $\fB_*$ with parameter $\ka$ such that $\ka(e) > 4000$ for all $e \in E$. Suppose that:
		\begin{enumerate}
			\item There exists some connected set $V \sset \Z^2$ such that $\del V \sset \fR(\infty)$.
			\item Let $\fB_\diamond = \fB_* \cap V$. Then $\fB_\diamond = \bigcup_{k=1}^\infty A_k(\fB_\diamond)$.
			\item $[\fD] \sset V \smin (\del V \cup \{0\})$, where $\fD := \bigcup_{k=1}^\infty D(A_k(\fB_\diamond), r_k/100)$.
		\end{enumerate} 
		Then $\fB(\infty) \sset [\fC] \cup V^c$, where $\fC :=\bigcup_{k=1}^\infty D(A_k(\fB_\diamond), 100 r_{k-1})$.
	\end{theorem}

\begin{remark}
	It may be helpful for the reader to consider the simpler special case of Theorem \ref{T:BssetI} when $V = \Z^2$. In this case, the condition $1$ is trivial, $\fB_\diamond = \fB_*$, and condition $3$ amounts to the condition that $0 \notin [\fD]$. In applications, we require the stronger version above to deal with the possibility that our process behaves badly near the origin.
\end{remark}

See Figure \ref{fig:multiscale} for a heuristic for the proof of Theorem \ref{T:BssetI}.

\begin{figure}
	\begin{center}
			\includegraphics[width=\textwidth]{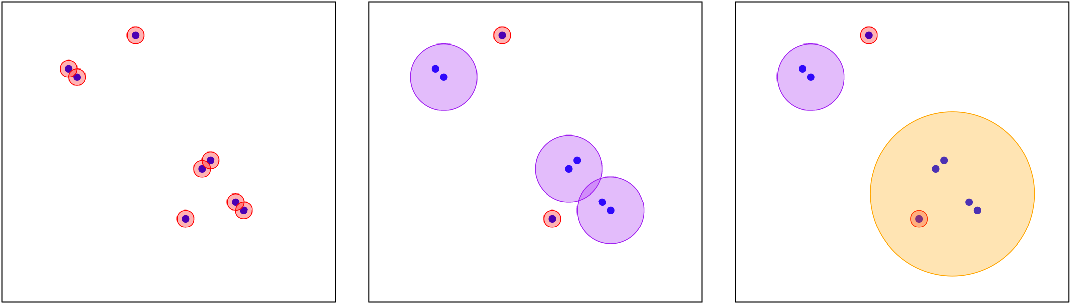}
		\caption{A heuristic for the multi-scale argument in Theorem \ref{T:BssetI}. Our goal is define to a region of $\Z^2$ that the red process will necessarily occupy. First draw a small ball around each blue seed, Panel A. If balls at this scale are isolated, then we can guarantee that the blue process will not spread beyond these balls. However, if balls overlap, then two blue seeds can work together to capture more area. Because of this we need to draw larger balls around these seeds at a higher scale, Panel B. Again, if balls at this higher scale overlap then we are forced to draw even larger balls, Panel C. In our setup, we are able to ignore the interaction between the orange ball and the red ball that was dealt with at an earlier scale. The precise details of our setup are slightly different than this heuristic, but the basic idea is the same.}
		\label{fig:multiscale}
	\end{center}

\end{figure}
	\subsection{Inputs for Theorem \ref{T:BssetI}}
	\label{S:inputs}
	
	The proof of Theorem \ref{T:BssetI} is based on a fairly involved inductive argument. In this section, we prove a collection of basic inputs that are required for the induction. The first input is a structural lemma about $\fC$ and $\fD$.
	We first define $k$-scale approximations
	\begin{equation*}
\fD_k := \bigcup_{i=1}^k D(A_i(\fB_\diamond), r_i/100) \qquad \text{ and } \qquad \fC_k := \bigcup_{i=1}^k D(A_i(\fB_\diamond), 100 r_{i-1}).
	\end{equation*}
	We will also need to work with a slight enlargement of $\fC$. Define
		\begin{equation*}
\fC^+ := \bigcup_{i=1}^\infty D(A_i(\fB_\diamond), 300 r_{i-1}),
	\end{equation*}
	and similarly define $\fC^+_k$. We will use throughout that $\fC_k \sset \fC_k^+ \sset \fD_k$, so when bounding the sizes of any of these sets, it will always suffice to bound $\fD_k$.
	\begin{lemma}
		\label{L:IJlemma}
		For every $k$, all components of $\fD_k, [\fD_k]$ have diameter at most $r_k$.
	\end{lemma}
	
	\begin{proof}
	Filling in holes cannot increase the diameter of a set, so it suffices to prove the lemma for $\fD_k$. We prove this by induction on $k$. The definition of $A_1(\fB_\diamond)$ implies that all components of $\fD_1$ are disks of radius $r_1/100$, which have diameter bounded above by $r_1$. Now suppose that the claim holds for $\fD_{k-1}$. Let $C$ be a connected component of $\fD_k$. We can write 
	$$
	C = C_k \cup C_{k-1},
	$$
	where $C_k \sset D(A_k(\fB_\diamond), r_k/100)$ and $C_{k-1} \sset \fD_{k-1}$. If $C_k$ has at least two distinct connected components, then there must be a connected component $C_{k-1}'$ of $C_{k-1}$, and connected components $C_{k, 1}, C_{k, 2} \sset C_k$ such that
	$$
	C_{k, 1} \cup C_{k, 2} \cup C_{k-1}'
	$$
	is connected. In particular, there must be points $x_1, x_2 \in A_k(\fB_\diamond)$ such that 
	$$
	d(x_1, C_{k-1}'), d(x_2, C_{k-1}') \le r_k/100 + 1. 
	$$
	Since $C_{k-1}'$ has diameter at most $r_{k-1}$ by the inductive hypothesis, this implies
	$$
	|x_1 - x_2| \le r_k/50 + 2 + r_{k-1}.
	$$
	On the other hand, $|x_1 - x_2| \ge r_k/50,$ since otherwise $D(x_1, r_k/100)$ and $D(x_2, r_k/100)$ would overlap and hence so would $C_{k, 1}, C_{k, 2}$. Therefore
	$$
	x_2 \in A(x_1, r_k/50, r_k/50 + 2 + r_{k-1}).
	$$
	The annulus above is contained in $A(x_1, r_{k-1}, r_k/3)$ for all $k$ since $10^{12} r_{k-1} \le r_k$ for all $k$, and $r_0 = 1$. This contradicts that both $x_1, x_2 \in A_k(\fB_\diamond)$. 
	
	Therefore $C_k$ consists of a single connected component. By the construction of $A_k(\fB_\diamond)$, we than have
	$$
	C_k = D(B^*, r_k/100),
	$$
	where $B^*$ is a cluster of blue seeds satisfying $\text{diam}(B^*) \le r_{k-1}$. Therefore 
	$$
	\text{diam}(C_k) \le r_k/50 + r_{k-1}.
	$$
	Attaching $C_k$ to connected components of $C_{k-1}$ can only increase the diameter by $2 r_{k-1}$ by the inductive hypothesis, so $
	\text{diam}(C) \le r_k/50 + 3r_{k-1} \le r_k,
	$
	as desired.
	\end{proof}

A key part of the inductive analysis involves showing that the red and blue processes cannot be slowed down or sped up too much by the presence of lower-scale blue clusters. We handle this issue by bounding lengths of paths that are forced to avoid $[\fC_k^+]$, and by bounding lengths of paths that can pass through components of  $[\fC_k^+]$ at no cost.

For a subset $G \sset \Z^2$, let 
$$
|y - z|_G
$$
be the graph distance between $y$ and $z$ in $G$, where the edge set is induced from the edge set $E$ on $\Z^2$. Also let
$$
|y - z|_{G+}
$$
be the weighted graph distance between two points in $\Z^2$, where all edges in $E$ have weight $1$, except for edges with at least one vertex in $G$, which get weight $0$. In the remainder of this section, we build up to the following proposition.

\begin{prop}
	\label{P:distance-bds}
	Define
	$$
	\ga_k = \prod_{i=1}^k \left(1 + \frac{10^{12} r_{i-1}}{r_{i}} \right),
	$$
	For every $k \ge 1$, for every $x, y \in [\fD_k]^c$ we have
	$$
	\ga_k^{-1}|x-y| \le |x-y|_{[\fC^+_k]+} \le |x-y|_{[\fC^+_k]^c} \le \ga_k |x-y|.
	$$
\end{prop}

In the statement above, recall that $[G]$ is equal to $G$ with its holes filled in.

The proof of Proposition \ref{P:distance-bds} relies on an inductive argument. Building to get the inductive step will require several lemmas. For these lemmas, we assume that we have the following setup. First, for a set $A \sset \Z^2$, define its \textbf{outer boundary}
$$
\del^o A := \del A^c.
$$
We can think of the usual boundary $\del A$ as an \emph{inner} boundary.
	\begin{enumerate}
	\item We have sets $C \sset D \sset \Z^2$ and $C' \sset D' \sset \Z^2$ with $\del^o D' \cap D = \emptyset$. Moreover, these sets have simply connected components.
	\item For some $r \in \N$, all components of $D \cup D'$ have diameter at most $r$.
	\item For some $r' > 1200r$, for any $x, y$ in distinct components of $D'$ we have $|x - y| > r'$.
	\item For all $x, y \notin D$, we have
	$$
	\frac{1}{2} |x - y| \le |x-y|_{C+} \le |x - y|_{C^c} \le 2|x - y|.
	$$
	\item For all $x, y$ in the same connected component of $\del^o D'$, we have
	$$
	|x - y|_{[C \cup C']^c} \le 10 |x - y|_{C^c}.
	$$
\end{enumerate}

\begin{lemma}
	\label{L:distance-bound'}
	Suppose that we are in the setup above, and that additionally, $D'$ consists of a single component.
	Then for any  $x, y \in [D(D', m) \cup D]^c$, we have
	$$
|x-y|_{[C \cup C']^c} \le \lf(1 + \frac{60 r}m \rg)|x - y|_{C^c}.
$$
	Similarly, for all $x,y \in [D(D', m) \cup D]^c$ we have
	$$
	|x-y|_{[C \cup C']+} \ge \lf(1 - \frac{3r}m \rg)|x - y|_{C+}.
	$$
\end{lemma}

\begin{proof}
	For the first implication, let $x, y \in [D(D', m) \cup D]^c$ and let $G(x, y)$ be a $|\cdot|_{C^c}$-geodesic from $x$ to $y$. If $G(x, y) \cap D' = \emptyset$, then $|x-y|_{[C \cup C']^c} = |x-y|_{C^c}$, and the claim follows. Also, if
	$$
	2|x - y| \le d(x, D') + d(D', y), 
	$$
	then any path from $x$ to $y$ that goes through $D'$ is at least length $2|x - y|$ which is at least $|x-y|_{C^c}$ by assumption $4$. Hence $G(x, y) \cap D' = \emptyset$ in this case as well.
	
	It just remains to prove the lemma when
	\begin{equation}
	\label{E:2xyxy}
	2|x - y| \ge d(x, D') + d(D', y)
	\end{equation}
	and $G(x, y) \cap D' \ne \emptyset$. We can find $a \ne b \in \del^o D' \cap G(x, y)$ such that $G(x, a) \cap D'=\emptyset, G(b, y) \cap D' = \emptyset$, where $G(x, a), G(b, y)$ are the segments of $G(x, y)$ from $x$ to $a$ and $b$ to $y$.  Now, by the triangle inequality we have that
	\begin{align}
	\label{E:xab1}
	|x -y|_{[C \cup C']^c} &\le |x-a|_{C^c} + |a-b|_{[C \cup C']^c} + |b- y|_{C^c}.
	\end{align}
	Next, we bound $|a-b|_{[C \cup C']^c}$. First, since $a, b \in \del^o D'$, neither point lies in $D$ by assumption $1$. Therefore we can apply the bound in assumption $4$. Combining this with the bound in assumption $5$ gives that
	$$
	|a-b|_{[C \cup C']^c} \le 20 |a-b|.
	$$
	By assumption $2$, this is bounded above by $20 (r+2) \le 60 r$, so \eqref{E:xab1} is bounded above by $|x - a|_{C^c} + |b - y|_{C^c} + 60r$, which is in turn bounded above by
	\begin{align}
	\label{E:3xyxy}
	&|x-y|_{C^c} + 60r.
	\end{align}
	For this bound, we have used that
	$$
	|x - y|_{C^c} = |x- a|_{C^c} + |a-b|_{C^c} + |b - y|_{C^c} \ge |x- a|_{C^c} + |b - y|_{C^c}. 
	$$ 
	Now, \eqref{E:2xyxy} and the fact that $x, y \notin D(D', m)$ implies that $|x -y|_{C^c} \ge m$. Combining this with \eqref{E:3xyxy} yields the bound.
	
	For the second implication,	let $G^+(x, y)$ be a $|\cdot|_{[C \cup C']+}$-geodesic from $x$ to $y$. Again, the bound is straightforward if $G^+(x, y) \cap D' = \emptyset$, and if not we can find $a \ne b \in \del^o D' \cap G^+(x, y)$ such that $G^+(x, a) \cap D'=\emptyset, G^+(b, y) \cap D' = \emptyset$. Therefore
	\begin{align*}
	|x-y|_{[C \cup C']+} & \ge |x-a|_{C+} +|b - y|_{C+}
	\end{align*}
	Now, on the other hand the triangle inequality implies
	\begin{align*}
|x-y|_{C+} &\le |x-a|_{C+} +|a-b|_{C+} + |b - y|_{C+} \\
&\le |x-a|_{C+} +|a-b| + |b - y|_{C+} \\
&\le |x-a|_{C+} + r + 2 + |b - y|_{C+},
	\end{align*}
	so setting $Z =  |x-a|_{C+} +|b - y|_{C+}$, the ratio of $|x-y|_{[C \cup C']+}$ and $|x-y|_{C+}$ is at least $Z/(Z + r + 2)$. Finally,  $Z \ge m$ since $x, y \notin D(D', m)$, so $Z/(Z + r + 2) \ge m/(m + r + 2) \ge 1 - (3r)/m$, yielding the bound.
\end{proof}

\begin{lemma}
\label{L:d-bd-2}
We work in the same setup as in Lemma \ref{L:distance-bound'}, except without the additional assumption that $D'$ consists of a single component. Let $\hat D$ be any component of $D'$, and let $x, y \in A(\hat D, r'/10, r'/5) \cap D^c$. Then
\begin{equation}
\label{E:1-10-r}
\lf(1 - \frac{30 r}{r'} \rg)|x - y|_{C+} \le |x-y|_{[C \cup C']+} \le |x-y|_{[C \cup C']^c} \le \lf(1 + \frac{600r}{r'} \rg)|x - y|_{C^c}.
\end{equation}
Also, if $\tilde D$ is any component of $D$ not entirely contained in $D(D', r'/10)$, and $x, y \in \del^o \tilde D$, then
$$
|x - y|_{C+} = |x-y|_{[C \cup C']+}, \qquad |x-y|_{[C \cup C']^c} = |x - y|_{C^c}.
$$
\end{lemma}

\begin{proof}
Let $x, y \in A(\hat D, r'/10, r'/5) \cap D^c$. Let $\hat C = C' \cap \hat D$ be the part of $C'$ in $\hat D$. By Lemma \ref{L:distance-bound'} with $m = r'/10$, we have
\begin{equation}
\label{E:frac30r}
\lf(1 - \frac{30 r}{r'} \rg)|x - y|_{C+} \le |x-y|_{[C \cup \hat C]+} \le |x-y|_{[C \cup \hat C]^c} \le \lf(1 + \frac{10(60r)}{r'} \rg)|x - y|_{C^c}.
\end{equation}
Now, the assumptions on $r, r'$ ensure that both of the $(1 \pm \cdot)$-prefactors on the left and right above are in the interval $[2/3, 3/2]$. Combined with assumption 4 on $|\cdot|_{C+},|\cdot|_{C^c}$, we get that
$$
\frac{1}{3} |x - y| \le |x - y|_{[C \cup \hat C]+} \le |x-y|_{[C \cup \hat C]^c} \le 3|x - y|.
$$
Now, $d(x, \hat D), d(y, \hat D) \le r'/5$ and $\hat D$ has diameter less than $r$, so again the relationship between $r$ and $r'$ ensures that 
$$
|x-y| \le r'/2.
$$
On the other hand, since the components of $D'$ are spaced at least distance $r'$ apart, any path $P$ from $x$ to $y$ that enters another component of $D'$ must have length at least $8r'/5$. Since $|x-y|_{[C \cup \hat C]^c} \le 3|x - y| \le 3r'/2$, this implies that any $[C \cup \hat C]^c$-geodesic from $x$ to $y$ does not enter any other component of $D'$, and hence is also a $[C \cup C']^c$-geodesic. This proves the first inequality in \eqref{E:1-10-r}.

For the final inequality in \eqref{E:1-10-r}, we check that any $[C \cup C']+$-geodesic from $x$ to $y$ avoids $D' \smin \hat D$. Suppose that this is not the case. Then there must be points $z, z'$ in the outer boundary of $D' \smin \hat D$ such that 
$$
|x-y|_{[C \cup C']+} \ge |x - z|_{[C \cup \hat C]+} + |z' -y|_{[C \cup \hat C]+}.
$$
By assumption $1$ and assumption $3$, we have that $z, z' \notin D \cup D(\hat D, r'/10)$. Therefore by another application of Lemma \ref{L:distance-bound'} as in \eqref{E:frac30r} and assumption $4$ on $|\cdot|_{C+}$, the right hand side above is bounded below by
$$
\frac{2}{3}\lf(|x - z|_{C+} + |z' -y|_{C+}\rg) \ge \frac{1}{3}\lf(|x - z| + |z' -y|\rg) \ge 8r'/15.
$$
On the other hand, $|x-y|_{[C \cup C']+} \le |x -y| < 8r'/15$, which is a contradiction. Therefore any $[C \cup C']+$-geodesic from $x$ to $y$ avoids $D' \smin \hat D$, completing the proof of \eqref{E:1-10-r}.

The `Also' statement is similar. Since $\tilde D$ is not entirely contained in $D(D', r'/10)$ and has diameter at most $r$, for $x, y \in \del^o \tilde D$, we have $x, y \notin D(D', r'/20)$. Moreover, by assumption $4$, we have $|x-y|_{C^c} \le 2|x-y| \le 2(r+2) \le 6r$. Therefore any $|\cdot|_{C^c}$-geodesic from $x$ to $y$ will be too short to reach a point in $D'$, so $|x-y|_{C^c} = |x-y|_{[C \cup C']^c}$. Also, by assumption $4$, any path from $x$ to $y$ that reaches a point in $D'$ must have $(C \cup C')+$-length at least
$$
\frac{1}{2}\lf(d(x, D') + d(y, D') \rg) \ge r'/20 \ge |x - y|,
$$
so  $|x-y|_{C+} = |x-y|_{[C \cup C']+}$.
\end{proof}

\begin{lemma}
	\label{L:distance-bound}
	In the setup of Lemma \ref{L:d-bd-2}, for any $x, y \notin D(D', r'/10) \cup D$, we have
\begin{equation}
\label{E:frac10r}
\lf(1 - \frac{30 r}{r'} \rg)|x - y|_{C+} \le |x-y|_{[C \cup C']+} \le |x-y|_{[C \cup C']^c} \le \lf(1 + \frac{600 r}{r'} \rg)|x - y|_{C^c}.
\end{equation}
\end{lemma}

\begin{proof}
First let $A$ be a component of $D(D', r'/10) \cup D$. We first show that \eqref{E:frac10r} holds for $x, y \in \del^o A$. First suppose that $A \cap D(D', r'/10) \ne \emptyset$. Since all components of $D$ have radius at most $r$, and components of $D'$ are separated by distance $r'$, the set $A$ is also a component of $D(\hat D, r'/10) \cup D$ for some component $\hat D$ of $D'$ and $A$ is contained in $D(\hat D, r'/5)$. Therefore by the first part of Lemma \ref{L:d-bd-2}, all bounds in \eqref{E:frac10r} hold for $x, y \in \del^o A$. If $A \cap D(D', r'/10) = \emptyset$, then $A$ is a component of $D$ that does not intersect $D(D', r'/10)$, so \eqref{E:frac10r} holds by the second part of Lemma \ref{L:d-bd-2}.

Now consider arbitrary points $x, y \notin D(D', r'/10) \cup D$, and let $G(x, y)$ be a $|\cdot|_{C^c}$-geodesic from $x$ to $y$. For points $a, b \in G(x, y)$, we write $G(a, b)$ for the segment of $G(x, y)$ from $a$ to $b$.
 If $G(x, y) \cap (D(D', r'/10) \cup D) = \emptyset$, then $|x-y|_{(C \cup C')^c} = |x - y|_{C^c}$ and the upper bound in \eqref{E:frac10r} is immediate. If not, then there exists $k \in \N$ and points $a_1, b_1,\dots, a_k, b_k \in \del^o (D(D', r'/10) \cup D) \cap G(x, y)$ such that 
	\begin{enumerate}[label=(\roman*)]
		\item $G(x, a_1) \cap [D(D', r'/10) \cup D]=\emptyset, G(b_k, y) \cap [D(D', r'/10) \cup D] = \emptyset,$
		\item $G(b_i, a_{i+1}) \cap [D(D', r'/10) \cup D] =\emptyset$ for all $i$,
		\item $a_i, b_i$ are in the boundary of same component of $[D(D', r'/10) \cup D]$.
	\end{enumerate}
	Now by the triangle inequality,
	\begin{align*}
	|x -&y|_{[C \cup C']^c} \\&
	\le |x-a_1|_{[C \cup C']^c} + \sum_{i=1}^k|a_i-b_i|_{[C \cup C']^c} + \sum_{i=1}^{k-1} |b_i-a_{i+1}|_{[C \cup C']^c} + |b_k- y|_{[C \cup C']^c}\\
	&= |x-a_1|_{C^c} + \sum_{i=1}^k|a_i-b_i|_{[C \cup C']^c} + \sum_{i=1}^{k-1} |b_i-a_{i+1}|_{C^c} + |b_k- y|_{C^c} \\
	&\le |x-a_1|_{C^c} + \sum_{i=1}^k\lf(1 + \frac{600r}{r'}\rg)|a_i-b_i|_{C^c} + \sum_{i=1}^{k-1} |b_i-a_{i+1}|_{C^c} + |b_k- y|_{C^c} \\
	&\le \lf(1 + \frac{600r}{r'}\rg)|x - y|_{C^c},
	\end{align*}
	where in the fourth line, we have used the bound \eqref{E:frac10r} for points in the outer boundary of the same component of $D(D', r'/10) \cup D$, and in the final line we have used the geodesic property 
	$$
	|x - y|_{C^c} = |x-a_1|_{C^c} + \sum_{i=1}^k|a_i-b_i|_{C^c} + \sum_{i=1}^{k-1} |b_i-a_{i+1}|_{C^c} + |b_k- y|_{C^c}
	$$ 
	For the second implication, let $G^+(x, y)$ be a $|\cdot|_{(C\cup C')+}$-geodesic from $x$ to $y$.  Again, the bound is straightforward if $G^+(x, y) \cap (D(D', r'/10) \cup D) = \emptyset$, and if not we can find $a_1, b_1,\dots, a_k, b_k$ satisfying conditions (i)-(iii) above. Therefore similarly
	\begin{align*}
	|x-y|_{[C\cup C']+} &= |x-a_1| + \sum_{i=1}^k|a_i-b_i|_{[C\cup C']+} + \sum_{i=1}^{k-1} |b_i-a_{i+1}| + |b_k- y| \\
	&\ge |x-a_1| + \sum_{i=1}^k(1 - 30r/r')|a_i-b_i| + \sum_{i=1}^{k-1} |b_i-a_{i+1}| + |b_k- y| \\
	&\ge (1-30r/r')|x -y|. \qedhere
	\end{align*}
\end{proof}

The next lemma and corollary is concerned with finding a set $D'$ to satisfy the above conditions given sets $C, D,$ and $C'$. We will also use this next lemma later on to establish containment of blue processes.
\begin{lemma}
\label{L:separating-set}
Let $C \sset D \sset \Z^2, x \in \Z^2$ and $r \ge 1$. Suppose that all components of $D$ have diameter at most $r$, and that for any $y, z \in D^c$ we have
\begin{equation}
\label{E:11ga-bd}
|y - z|_{C^c} \le 2|y - z|.
\end{equation}
Let $A$ be the connected component of $D(x, 98 r) \cup D$ containing $x$. Then $\del^o A \sset D(x, 100 r)$ and for any $y, z \in \del^oA$ we have 
$$
|y-z|_{D(x, 200 r) \smin [C \cup D(x, 4r)]} \le 10|y-z|.
$$
\end{lemma}

\begin{proof}
The fact that $\del^o A \sset D(x, 100 r)$ is immediate since components of $D$ have diameter at most $r$.	
For $y, z \in \del^o A$, let $G(y, z)$ denote a $|\cdot|_{C^c}$-geodesic from $y$ to $z$. By \eqref{E:11ga-bd}, as long as $|y-z| < 94 r$, then $|y-z|_{C^c} < 188 r$, and so the geodesic $G(y, z)$ cannot enter $D(x, 4r)$ and will never exit $D(x, 200r)$. Therefore
$$
|y-z|_{D(x, 200 r) \smin [C \cup D(x, 4r)]} = 	|y-z|_{C^c} \le 2 |y-z|.
$$
Now suppose $|y-z| >94r$ for $y, z \in A$. Since $A$ is connected and contained in the annulus $A(x, 98r, 100r)$, for any two points $y, z \in A$, we can find points $z_0 = y, z_1, z_2, \dots, z_5 = z  \in \del^o A$ such that $|z_i - z_{i+1}| < 94r$. Hence by the triangle inequality,
\[
|y-z|_{D(x, 200 r) \smin [C \cup D(x, 4r)]} \le \sum_{i=1}^5 |z_{i-1} - z_i|_{C^c} \le 10\cdot94 r \le 10 |y-z|. \qedhere
\]
\end{proof}

\begin{proof}[Proof of Proposition \ref{P:distance-bds}]
We proceed by induction on $k$ for $k = 0, 1, \dots$, where $\fC^+_0 = \fD_0 = \emptyset$, and $\ga_0 = 1$. With these definitions, the $k=0$ base case is trivially true. Now suppose that the bounds hold for $k-1$.

First, we can find a set of points $\{b_i : i \in N\} \sset A_k(\fB_\diamond)$, where $N$ is a finite or countable index set such that
$$
|b_i - b_j| > r_k/3
$$
for all $i, j \in N$, and 
$$
A_k(\fB_\diamond) \sset \bigcup_{i \in N} D(b_i, r_{k-1}).
$$
This decomposition shows that points in $A_k(\fB_\diamond)$ can be grouped into clusters around each of the points $b_i$. 
We can apply Lemma \ref{L:separating-set} with $x=b_i, r = 100r_{k-1}$ and $C = [\fC^+_{k-1}], D = [\fD_{k-1}]$ for each $i$. The hypotheses of that lemma are satisfied by combining the inductive hypothesis with Lemma \ref{L:IJlemma}. Therefore letting $A_i$ be the connected component of $[\fD_{k-1}]\cup D(b_i, 98(100 r_{k-1}))$,for any $x, y\in \del^o A_i$ we have
\begin{equation}
\label{E:xyxyx}
|x - y|_{D(b_i, 200(100 r_i)) \smin [[\fC^+_{k-1}]\cup D(b_i, 4\cdot 100r_{k-1})]} \le 10 |x-y|.
\end{equation}
Now, since $|b_i - b_j| > r_k/3 > 400(100 r_{k-1})$ for $i \ne j$, we have
$$
D(b_i, 200(100 r_{k-1})) \cap D(A_k(\fB_\diamond), 300 r_{k-1}) \sset  D(b_i, 4\cdot 100r_{k-1}).
$$
Therefore for $x, y \in \del^o A_i$, we have
$$
|x - y|_{[D(A_k(\fB_\diamond), 300 r_{k-1}) \cup [\fC^+_{k-1}]]^c} \le |x - y|_{D(b_i, 200(100 r_i)) \smin [[\fC^+_{k-1}] \cup D(b_i, 4\cdot 100r_{k-1})]}.
$$
Therefore by \eqref{E:xyxyx},
\begin{equation}
\label{E:XIk}
|x - y|_{[D(A_k(\fB_\diamond), 300 r_{k-1}) \cup [\fC^+_{k-1}]]^c} \le 10 |x-y| \le 10 |x-y|_{[\fC^+_{k-1}]^c}.
\end{equation}
We claim that we are now in the setup introduced prior to Lemma \ref{L:distance-bound'}, with
\begin{itemize}
	\item $C = [\fC^+_{k-1}], D = [\fD_{k-1}]$, $C'= D(A_k(\fB_\diamond), 300 r_{k-1})$ and $D' = \bigcup_{i \in N} A_i$,
	\item $r = 2\cdot 99\cdot 100 r_{k-1}$,
	\item $r' = r_k/200$.
\end{itemize}
Indeed, condition 1 is immediate from the construction. Condition 2 follows since all the $A_i$ have diameter at most $2 \cdot 99 \cdot 100 r_{k-1}$ and are disjoint.
The fact that $r_k \ge 10^{12} r_{k-1}$ ensures that $d(A_i, A_j) \ge r'$ for $i \ne j$, guaranteeing condition 3. The inductive hypothesis and the bound \eqref{E:gammabound} on $\ga$ ensures Condition 4. Condition 5 follows from \eqref{E:XIk}. Therefore by Lemma \ref{L:distance-bound} and the inductive hypothesis, for any $x, y \notin D(D', r'/10) \cup [\fD_{k-1}]$ we have
\begin{align*}
&\lf(1 - \frac{10 \cdot 2\cdot 99\cdot 100 r_{k-1}}{r_k/200} \rg)\ga_{k-1}^{-1}|x - y| \le |x-y|_{[C \cup C']+}, \\
&|x-y|_{[C \cup C']^c} \le \lf(1 + \frac{600 \cdot 2\cdot 99\cdot 100 r_{k-1}}{r_k/200} \rg)\ga_{k-1}|x - y|
\end{align*}
The prefactors in the upper and lower bounds above are bounded above and below by $\ga_k$ and $\ga_k^{-1}$. Moreover, it is easy to check that $D(D', r'/10) \cup [\fD_{k-1}] \sset [\fD_k]$ and $[\fC^+_k] = [C \cup C']$, so the proposition follows. 
\end{proof}

\subsection{Bounding blue growth}

To bound the blue growth in SSPs, we will divide our space into SSPs on finite subsets of $\Z^2$ on which we will be able to apply the bounds of Proposition~\ref{P:distance-bds}. We start with a lemma that will allow us to patch together different finite processes to get a bound on the growth of the entire process.

\begin{lemma}
\label{L:boundary-patching} Fix $\ka:E \to (0, \infty)$ and a collection of blue seeds $\fB_*$. Suppose that we can partition the set $\fB_*$ into clusters $B_1, B_2, \dots$ such that for each $i$, there exist finite sets $C_i, D_i$ with $B_i \sset C_i \sset D_i$ and:
\begin{itemize}
	\item The sets $D_i$ are all disjoint, and $\del^o C_i \sset D_i$ for all $i$.
	\item For an SSP $(\fR_i, \fB_i)$ on $D_i$ with collection of blue seeds $B_i$, parameter $\ka$, source function defined on $\del D_i$,  and any edge weights, we have 
	$$
	\fB_i(\infty) \sset C_i.
	$$
\end{itemize}
Consider a set $M$ with $\del M \cap D = \emptyset$, where $D := \bigcup D_i$. Then for any SSP $(\fR, \fB)$ on $M$ with parameter $\ka$ and arbitrary edge weights and source function defined on $\del M$, we have $\fB(\infty) \sset C := \bigcup C_i$. More generally, any SSP on all of $M$ started from a source function defined on a subset of $M\smin D$ with parameter $\ka$ satisfies $\fB(\infty) \sset C$.

\end{lemma}

\begin{proof}
	The two claims have identical proofs, so we handle them together.
	Suppose that the lemma fails. Let $t > 0$ be the first time when $\fB(t) \not\sset C$ and let $t' \in [t, \infty)$ be the first time when $\fB(t) \not\sset D$; we take $t' =\infty$ if no such time occurs. Let $Q = \fB(t) \smin C$.
No square in $Q$ can be a blue seed. Thus since $\ka > 0$ pointwise, every square in $Q$ must be adjacent to a square in $C$ and hence $\fB(t) = Q \cup C \sset D$ by the first bullet. Therefore $t < t'$.

Now let $s \in Q$, and let $D_i$ be the component of $D$ that $s$ is contained in.  On the interval $[0, t')$, the SSP $(\fR, \fB)$ restricted to $D_i$ is simply an SSP $(\fR_i, \fB_i)$ on $D_i$ with a source defined on $\del D_i$ and blue seed set $B_i$. This is because no part of the blue process can invade from $D_i^c$ at times in $[0, t')$, since this would imply that $\fB(r)$ is not contained in $D$ for some $r < t'$. However, the second bullet guarantees that $s \notin \fB_i(r)$ for all $r < t'$. Since $t < t'$, this is a contradiction.
\end{proof}

The next lemma provides the main inductive step in the proof of Theorem \ref{T:BssetI}.
\begin{lemma}
\label{L:higher-scales} Let $\ka:E \to (4000, \infty)$. All SSPs in this lemma have the same parameter $\ka$. Fix $x \in \Z^2$ and $r \in \N$. Suppose that we can partition the set of blue seeds $\fB_* = \fB_a \cup \fB_b$, where $\fB_a$ and $\fB_b$ have the following properties:
	
	\begin{enumerate}
		\item There are sets $C_1 \sset C_2 \sset C_3$ with $\fB_a \sset C_1 \sset C_2 \sset C_3$.
		\item All connected components of $C_3$ have diameter $\le r$.
		\item $\fB_b \sset D(x, r)$.
		\item $\del^o C_1 \sset C_2$.
		\item For $x, y \in C_3^c$, we have
		\begin{equation}
		\label{E:1ga-bd}
		\frac{1}{2}|x-y| \le |x-y|_{C_2+} \le |x - y|_{C_1^c} \le 2|x - y|.
		\end{equation}
		\item For any connected component $C$ of $C_2$, and any SSP $(\fR_C, \fB_C)$ on $C$ with parameter $\ka$ and source defined on $\del C$, we have
		$$
		\fB_C(\infty) \sset C_1.
		$$
	\end{enumerate}
	Now let $M$ be any connected set containing $D(x, 200r)$, such that $\del M \sset C^c_3$. Consider an SSP $(\fR, \fB)$ on $M$ with parameter $\ka$, collection of blue seeds $\fB_*$, and source defined on $\del M$. Then
	\begin{equation}
	\label{E:MsminC1}
	\fB(\infty) \sset C_1 \cup D(x, 100 r).
	\end{equation}
\end{lemma}

\begin{proof}
First, if $\fB_b = \emptyset$, then by assumption $6$, $\fB(\infty) \sset C_1$ via Lemma \ref{L:boundary-patching}. We now move to the general case. By Lemma \ref{L:separating-set} with $C_1 = C$ and $C_3 = D$ in that lemma, letting $A$ be the component of $D(x, 98 r) \cup C_3$ containing $x$, for any $x, y \in \del^o A$, we have	
	\begin{equation}
	\label{E:xyyc1}
	|x - y|_{M\setminus [C_1 \cup D(x, 4r])} \le 10 |x-y| \le 2000r.
	\end{equation}
	Here we have used that $M$ contains $D(x, 200r)$ to apply Lemma \ref{L:separating-set}. We use estimate \eqref{E:xyyc1} to prove Lemma \ref{L:higher-scales}.
	
	Let $\tau$ be the first time when a blue seed $b \in \fB_b$ gets added to the blue process $\fB$. Now, $\fB(t) \cup \fR(t) \cup \del M$ is connected for all $t$ and $\del^o A$ separates $\del M$ and $\fB_b$. Therefore for $b$ to join the blue process $\fB$, there must be a point $x \in \del^o A$ with $T(x) \le \tau$. Also, for $t <\tau$, we have $\fB(t) \sset C_1$ since the process evolves as in the $\fB_b = \emptyset$ case up to this time. Therefore since $C_1 \cap \del^o A = \emptyset$, we have $C(x) = \fR$.

We will check that $\fR$ will engulf every square in $\del^o A$ before the blue process leaves $D(x, 4r) \cup C_1$. By the estimate \eqref{E:xyyc1}, $\fR$ will engulf all squares in $A$ by time $T(x) + 2000r$ unless $\fB$ can exit $D(x, 4r) \cup C_1$ before time $T(x) + 2000r$. We suppose that this is not the case. Let $s_0$ be the location of the first vertex outside of $D(x, 4r) \cup C_1$ that joins $\fB$, and let $t_0 = T(s_0)$. If there are multiple potential choices for $s_0$, we choose one arbitrarily. 

Next, we will construct a chain of points from $s_0$ ending at a point $s^*$ which is adjacent to $D(x, r)$. First, for the square $s_0$ to be added to the blue process, there must be a chain of squares
$$
s_{0, 0} = s_0, s_{0, 1}, \dots, s_{0, k_0} = q_1, s_{0, k_0 + 1} = p_1
$$
such that
\begin{enumerate}
	\item $(s_{0, i}, s_{0, i+1}) \in E$ for all $i$ and all $s_{0, i}$ are coloured blue,
	\item $p_1$ is located in $C_1 \cup D(x, r)$, and $s_{0, j} \notin C_1 \cup D(x,r)$ for all $j \le k_0$.
	\item Every edge $(s_{0, i+1}, s_{0, i})$ is used in the construction of the
	SSP. More precisely, $T(s_{0, i+1}) + X_{C(s_{0, i+1})}(s_{0, i+1}, s_{0, i}) = T(s_{0, i})$ for all $i$.
\end{enumerate}
In particular, point 2 guarantees that $s_{0, j}$ is not a blue seed for any $j < k_0$, so point 3 guarantees that the edges $(s_{0, i+1}, s_{0, i})$ are all used with weight $\ka(s_{0, i+1}, s_{0, i}) > 4000$. Hence
\begin{equation}
\label{E:s0q1}
T(q_1) < T(s_0) - 4000 k_0 \le T(s_0) - 4000 |q_1 - s_0| \le T(s_0) - 4000 |q_1 - s_0|_{C_2 +}.
\end{equation}
At this point, if $p_1 \in D(x,r)$, we set $s^* = q_1$. If this is not the case we proceed as follows. Let $C$ be the component of $C_2$ containing $q_1$ so that $q_1 \in C \smin C_1$. By assumption 6 of the theorem, the only way that $q_1$ can be absorbed into the blue process is if some point $s_1 \in \del C$ was absorbed by the blue process in the time interval $[0, T(q_1))$. Therefore
\begin{equation}
\label{E:s1p1}
T(s_1) < T(q_1) = T(q_1) - 4000 |q_1 - s_1|_{C_2 +},
\end{equation}
where the equality uses that $q_1$ and $s_1$ are in the same component of $C_2$, so $|q_1 - s_1|_{C_2 +} = 0$.

We can continue in this way, constructing a sequence of points $s_0, q_1, s_1, q_2, s_2, \dots$, terminating when we reach a point $q_k := s^*$ which is adjacent to $D(x, r)$. Note that this procedure must terminate since the SSP on $M$ only infects finitely many squares. Now, adding together the inequalities in \eqref{E:s0q1} and \eqref{E:s1p1} for all $q_i, s_i$ and applying the triangle inequality in the $|\cdot|_{C_2+}$ metric, we have that
$$
T(s^*) < T(s_0) - 4000 |s^* - s_0|_{C_2+}.
$$
Now, $s_0 \notin D(x, 4r)$ and $s^* \in D(x, r + 1)$. We will use this to bound $|s^* - s_0|_{C_2 +}$. Let $G$ be a $|\cdot|_{C_2 +}$-geodesic between $s_0$ and $s^*$. Since $|s_0 - s^*| \ge 3r$ and all components of $C_3$ have diameter at most $r$, there must be points $z_1, z_2 \in G \smin C_3$ with $|z_1 - z_2| \ge r$. Therefore by \eqref{E:1ga-bd}, we have
$$
|s^* - s_0|_{C_2+} \ge |z_1 - z_2|_{C_2 +} \ge  r/2, \quad \text{ and so } \qquad T(s^*) \le T(s_0) - 2000 r.
$$
Now, $T(x) \le T(s^*)$ so we have 
$$
T(x)  + 2000 r < T(s_0).
$$
Since we have assumed that $T(s_0) < T(x) + 2000 r$, this is a contradiction.
Therefore the red process engulfs every square in $\del^o A$, and so the growth of blue process from $\fB_b$ is contained to $A$. Also, no component of $C_2$ on the exterior of $A$ can be invaded by the part of the blue process that can be linked back to $\fB_b$. Therefore by the $\fB_b = \emptyset$ case, the growth of the blue process outside of $A$ is contained to $C_1$. Combining these with the fact that $A \sset D(x, 100r)$ yields \eqref{E:MsminC1}.
\end{proof}

Finally, we have all the tools to prove Theorem \ref{T:BssetI}.
	\begin{proof}[Proof of Theorem \ref{T:BssetI}] First, if $0 \in V$, then let $\del^* V = \del V \cup \{0\}$. If not, let $\del^*V = \del V$. Define $S:\del^* V \to [0, \infty)$ by setting $S(u) = T(u)$. The assumptions of Theorem \ref{T:BssetI} guarantee that all $u \in \del^* V$ get coloured red, so the restriction of $(\fR, \fB)$ to $V$ is an SSP on $V$ with source function $S$. It suffices to analyze this process, denoted by $(\fR', \fB')$. We will show that for any $t$ we have $\fB'(t) \sset [\fC]$. 
		
		Since our SSP on all of $\Z^2$ has finite speed, by the second assumption of the theorem, at any time $t$ there is a maximal scale $k_t$ such that only blue seeds in sets $\fB_{\diamond, k}, k \le k_t$ have been encountered by $(\fR', \fB')$ by time $t$. Therefore $(\fR', \fB')$ grows in the same way up to time $t$ as if we replaced the collection of blue seeds by the smaller set $\bigcup_{k=1}^{k_t} A_k(\fB_{\diamond})$.
		
			Now, define $K_0 = \emptyset$, and let $K_k$ be the union of all connected components of $D(A_k(\fB_\diamond), 300r_{k-1}) \cup \fD_{k-1}$ containing a point in $A_k(\fB_\diamond)$, and let $\cK_k = \bigcup_{i=1}^k K_i$. Note that $\cK_k \sset [\fD_k]$ for all $k$.
		We will use induction on $0 \le k \le k_t$ to prove that
		\begin{itemize}
			\item For any component $C$ of $[\cK_k]$, any parameter-$\ka$ SSP $(\fR_C, \fB_C)$ on $C$ with collection of blue seeds $\fB_\diamond \cap C$ and source defined on $\del C$ has
			$$
			\fB_C(\infty) \sset [\fC] \cap C.
			$$
		\end{itemize} 
		
		The $k=0$ base case is trivial.
		Suppose that this claim holds for $k-1$, and consider a connected component $C$ of $[\cK_k]$. If $C$ does not contain an element of $A_k(\fB_\diamond)$, then $C$ was also a connected component of $[\cK_{k-1}]$, and the claim follows from the inductive hypothesis. 
		
		Therefore we may assume that $C$ does contain an element $x$ of $A_k(\fB_\diamond)$. 
		Components of $[\fD_{k-1}]$ have diameter at most $r_{k-1}$ by Lemma \ref{L:IJlemma}. By this bound and the construction of $A_k(\fB_\diamond)$, we necessarily have that $A_k(\fB_\diamond) \cap C \sset D(x, r_{k-1})$.
		
		At this point, we are in the setup of Lemma \ref{L:higher-scales}, with $M = C$ and
		\begin{itemize}
			\item $\fB_a = \bigcup_{i=1}^{k-1} A_i(\fB) \cap C, \qquad \fB_b = A_k(\fB_\diamond) \cap C \sset D(x, r_{k-1})$.
			\item $r = r_{k-1}$.
			\item $C_1 = [\fC_{k-1}], C_2 = [\cK_{k-1}], C_3 = [\fD_{k-1}].$
		\end{itemize} 
	Conditions $1, 3,$ and $4$ of Lemma \ref{L:higher-scales} follow from the above construction. Condition $2$ follows from Lemma \ref{L:IJlemma}. Condition $5$ follows from Proposition \ref{P:distance-bds}, the bound \eqref{E:gammabound} on $\ga$, and the fact that $[\cK_{k-1}] \sset [\fC^+_{k-1}]$. Condition $6$ follows from the inductive hypothesis. Therefore the claim follows from Lemma \ref{L:higher-scales}, since
		$$
		D(x, 100 r_{k-1}) \sset \fC_k.
		$$
		Now we apply Lemma \ref{L:boundary-patching}, where the blue seed set is $\bigcup_{k=1}^{k_t} \fB_{k, \diamond}$, the sets $C_i$ are components of $[\fC_{k_t}]$, the sets $D_i$ are components of $[\mathcal K_{k_t}]$, the set $M = V$, and the source is defined on $\del^* V \sset M \smin D$. The key hypothesis of that lemma is given by the above claim. Lemma \ref{L:boundary-patching} then implies that  $\fB'(t) \sset [\fC_{k_t}] \sset [\fC]$, as desired.
	\end{proof}

	\subsection{Bounds for random blue seeds}
	\label{SS:random-blue-seeds}
	
	So far, our study of SSPs has been purely deterministic. In the remainder of the paper, we want to consider random SSPs. Typically, the randomness on the edge weights will be quite complex. However, in our applications of Theorem \ref{T:BssetI} the configuration of blue seeds $\fB_\diamond$ will be always be stochastically dominated by a low-intensity i.i.d.\ process. In this section, we bound the behaviour of $\fC, \fD$ in this case. We start with a few straightforward lemmas. The first is a deterministic monotonicity lemma.

	\begin{lemma}
		\label{L:dom-seeds}
		Let $\fB_\diamond' \sset \fB_\diamond \sset \Z^2$ be two collections of blue seeds. Then for every $k \ge 1$, we have $\fB_{\diamond, k}' \sset \fB_{\diamond, k}$. Moreover, if the set $\fB_\diamond$ satisfies assumptions $2$ and $3$
		of Theorem \ref{T:BssetI} for some set $V$, then $\fB_\diamond'$ also satisfies these assumptions, and with $\fC, \fD, \fC', \fD'$ as in that theorem grown from the sets $\fB_\diamond, \fB_\diamond'$, we have
		\begin{equation}
		\label{E:IbIB}
		\fC' \sset \fC \qquad \text{ and } \qquad \fD' \sset \fD.
		\end{equation}
	\end{lemma}
	
	Note that the monotonicity in Lemma \ref{L:dom-seeds} does not necessarily hold at the level of the entire blue processes $\fB, \fB'$ generated from the seed sets $\fB_\diamond, \fB_\diamond'$.
	
	\begin{proof}
		We prove the first claim by induction. The base case $\fB_{\diamond, 1}' \sset \fB_{\diamond, 1}$ follows since $\fB_{\diamond, 1} = \fB_{\diamond}, \fB_{\diamond, 1}' = \fB_{\diamond}'$. Now assume that $\fB_{\diamond, k}' \sset \fB_{\diamond, k}$. By definition, we have $\fB_{\diamond, k+1} = \fB_{\diamond, k} \smin A_k(\fB_\diamond)$, and so by the inductive hypothesis,
		$$
		\fB_{\diamond, k}' \smin (A_k(\fB_\diamond) \cap \fB_{\diamond, k}') \sset \fB_{\diamond, k+1}.
		$$
		Since $\fB_{\diamond, k+1}' = \fB_{\diamond, k}' \smin A_k(\fB_\diamond')$, to complete the proof it suffices to show that 
		\begin{equation*}
		\label{E:AkB}
		A_k(\fB_\diamond) \cap \fB_{\diamond, k}' \sset A_k(\fB_\diamond'). 
		\end{equation*}
		This follows from the definition and the inductive hypothesis. Indeed,
		\begin{align*}
		A_k(\fB_\diamond') &= \{x \in \fB_{\diamond, k}' : \fB_{\diamond, k}' \cap A(x, r_{k-1}, r_k/3) = \emptyset\} \\
		&\supset \{x \in \fB_{\diamond, k}' : \fB_{\diamond, k} \cap A(x, r_{k-1}, r_k/3) = \emptyset\} \\
		&= A_k(\fB_\diamond) \cap \fB_{\diamond, k}',
		\end{align*}
		where the containment uses the inductive hypothesis. We move on to the `Moreover' statement. Assume that $\fB_\diamond, V$ satisfy the assumptions of Theorem \ref{T:BssetI}. Then by assumption $2$ in Theorem \ref{T:BssetI}, we have $\bigcap_{k \ge 1} \fB_{\diamond, k} = \emptyset$ and so $\bigcap_{k \ge 1} \fB'_{\diamond, k} = \emptyset$ as well. This implies assumption $2$ of Theorem \ref{T:BssetI} for $\fB_{\diamond}'$. It just remains to show \eqref{E:IbIB}, since this will also guarantee that $\fB_{\diamond}'$ satisfies assumption $3$ of Theorem \ref{T:BssetI}. To see why these containments hold, simply observe that since $\fB_{\diamond, k}' \sset \fB_{\diamond, k}$ for all $k$ and $\bigcap_{k \ge 1} \fB_{\diamond, k} = \emptyset$, that if $b \in A_k(\fB_{\diamond}')$ for some $k$, then $b \in A_{\ell}(\fB_{\diamond})$ for some $\ell \ge k$.
	\end{proof}

Next, we show that points in higher scales $\fB_{\diamond, k}$ are scarce if $\fB_\diamond$ is dominated by a low-intensity Bernoulli process.

	\begin{lemma}
		\label{L:Ak-estimate}
		Consider any collection of blue seeds $\fB_{\diamond}$. 
		Then for any set $X \sset \Z^2$, the set $X \cap \fB_{\diamond, k}$ is a function of $\fB \cap D(X, r_{k-1}/2)$. 
	
	Moreover, suppose that $\fB_{\diamond}$ is stochastically dominated by an i.i.d.\ Bernoulli process with mean $p> 0$.  There exists a universal constant $c > 0$ such that for any $x \in \Z^2$,
		$$
		\mathbb P(x \in \fB_{\diamond, k}) \le (c p)^{2^{k-1}}.
		$$
	\end{lemma}
	
	\begin{proof}
		We prove the first statement by induction on $k$. The $k=1$ base case is immediate. Now let $k \ge 2$ and suppose that the statement holds for $k-1$. By the definition of $A_k(\fB_\diamond)$, the event $x \in \fB_{\diamond, k}$ only depends on $\fB_{\diamond, k-1} \cap D(x, r_{k-1}/3)$ and whether or not $x \in \fB_\diamond$. By the inductive hypothesis, the set $\fB_{\diamond, k-1} \cap D(X, r_{k-1}/3)$ only depends on $\fB_{\diamond} \cap D(X, r_{k-1}/3 + r_{k-2})$. The claim follows since $r_{k-1}/3 + r_{k-2} < r_{k-1}/2$.
		
	Now we move to the `Moreover' claim. By Lemma \ref{L:dom-seeds}, we may assume that $\fB_{\diamond}$ is a Bernoulli process with mean $p$.
	 Let $\rho_k$ be the probability that $x \in \fB_{\diamond, k}$. We recursively bound $\rho_k$. If $x \in \fB_{\diamond, k}$, then $x \in \fB_{\diamond, k-1}$ but $x \notin A_{k-1}(\fB_{\diamond})$. Therefore there is some other $y \in \fB_{\diamond, k-1}$ such that $r_{k-2} < |x - y| \le r_{k-1}/3$. By the first part of the lemma, points in $\fB_{\diamond, k-1}$ are independent at distances greater than $r_{k-2}/2$. Therefore the probability that this happens for a fixed $y$ is $\rho_{k-1}^2$. Therefore by a union bound,
	 $$
	 \rho_k \le r_{k-1}^2 \rho_{k-1}^2.
	 $$
	Using the initial bound
		$
		\rho_1 \le p,
		$
		we get that
		$$
		\rho_k \le p^{2^{k-1}} r_1^{2^{k-1}} r_3^{2^{k-2}} \cdots r_{k-1}^{2}.
		$$
		Now, $r_k \le c e^{2k \log k}$ for all $k$ for some $c > 0$, so we have that
		$$
		\rho_k \le p^{2^{k-1}} c^{2^k}\exp \lf(2^{k-1} \sum_{i=1}^k 2^{-i} i \log i \rg).
		$$
		The sum $ \sum_{i=0}^\infty 2^{-i} i \log i$ is convergent, so the above is bounded by
		$
		(cp)^{2^{k-1}}
		$
		for some universal $c >0$.
	\end{proof}
	
	We use Lemma \ref{L:Ak-estimate} to get bounds on the engulfing set $\fD$. 
	\begin{lemma}
		\label{L:largest-scale} Again, assume that $\fB_{\diamond}$ is stochastically dominated by an i.i.d.\ Bernoulli process with parameter $p> 0$.
		Fix a ball $D(x, r)$ for some $r \in \N$, and let $M(x, r)$ be the diameter of the largest component of $[\fD]$ that intersects $D(x, r)$. For some absolute $c>0$, we have
		\begin{equation}
		\label{E:Mrrk}
		\P\lf( M(x, r) > r_k \rg) \le r^2 (cp)^{2^{k-1}} \qquad \text{ and } \qquad \P\lf( M(x, r) > 0 \rg) \le r^2 cp.
		\end{equation}
		In particular, for every $n_0 \in \N$ with $n_0 \ge 3$ we have
		\begin{equation}
		\label{E:Mn-BC}
		\P\lf( M(x, n) \ge n/2 \text{ for some } n \ge n_0\rg) \le  \exp\lf(\log(cp) \exp (\log n_0/(c\log \log n_0))\rg).
		\end{equation}
		Moreover, for all $k \ge 1$,
		\begin{equation}
		\label{E:fDk-estimate}
		\P([\fD] \cap D(x, r) \ne [\fD_k] \cap D(x, r)) \le  r^2 (cp)^{2^{k-1}},
		\end{equation}
		and for all sufficiently small $p$, we have
		\begin{equation}
		\label{E:PDxr}
		\P\lf(|[\fD] \cap D(x, r)| \ge \frac{1}{3}|D(x, r)| \rg) \le \exp\lf(\log(cp) \exp (\log r/(c\log \log r))\rg)
		\end{equation}
		for an absolute constants $c > 0.$
		The same bounds hold with $\fC, \fC_k$ in place of $\fD, \fD_k$.
	\end{lemma}

	\begin{proof}
	First, $[\fD] = \bigcup_{k=1}^\infty [\fD_k]$. This follows since any bounded component of $\fD^c$ is a bounded component of $\fD_k^c$ for all large enough $k$.	Next, by Lemma \ref{L:IJlemma}, all components of $[\fD_k]$ have diameter at most $r_k$. Therefore letting $[\fD_k]_a$ be the collection of components of $[\fD_k]$ that are entirely contained in $D(x, a)$ and are not also components of $[\fD_{k-1}]$, we have 
	$$
	[\fD] \cap D(0, r) \sset \bigcup_{k=1}^\infty [\fD_k]_{r + r_k},
	$$
	and
	$$
	M(r) \le r_K, \text{ where } K = \sup \{k \ge 1: [\fD_k]_{r + r_k} \ne \emptyset\}.
	$$ 
	Moreover, if the set above is empty, then $M(r) = 0$. We estimate $\P([\fD_k]_{r + r_k} \ne \emptyset)$. This is bounded above by $\P(A_k(\fB_\diamond) \cap D(x, r + r_k) \ne \emptyset)$. We can bound this by Lemma \ref{L:Ak-estimate}, a union bound, and the bound $r_k \le c\exp(2k \log k)$. This yields \eqref{E:Mrrk}. The bound \eqref{E:fDk-estimate} also follows from this proof, and the analogous bound for $\fC, \fC^k$ follows analogously.
	
	The inequality \eqref{E:Mn-BC} follows from the first inequality in \eqref{E:Mrrk} and a union bound over $n \in \{n_0, n_0 + 1, \dots\}$. To apply \eqref{E:Mrrk}, we use the fact that $r_k \le \exp (3k \log k)$ for all large enough $k$, which ensures that for large enough $n$ we have $n/2 > r_{k(n)}$, where $k(n) = \ceil{\log n/(4\log \log n)}$. Note that to go from \eqref{E:Mrrk} to \eqref{E:Mn-BC} we may need to increase $c$.
	
	For \eqref{E:PDxr}, we first bound the same probability with $[\fD_k]$ in place of $[\fD]$. By the second bound in \eqref{E:Mrrk} with $r=1$ and translation invariance of each of the sets $[\fD_k]$, we have that 
	$$
	\E |[\fD_k] \cap D(x, r)| \le cp |D(x, r)|.
	$$
	Moreover, the independence in Lemma \ref{L:Ak-estimate} and the diameter bound in Lemma \ref{L:IJlemma}  implies that if points $x_1, \dots, x_k$ are distance at least $3r_k$ apart, then the events $\{x_i \in [\fD_k]\}$ are independent. Therefore by partitioning into $3 r_k \times 3 r_k$ blocks that cover $D(x, r)$, for $r_k \le 10 r$ we can write 
	$$
	|[\fD_k] \cap D(x, r)| \le  \sum_{i=1}^{9r_k^2} \sum_{j=1}^{L}  X_{i, j},
	$$
	where the $X_{i, j}$ are identically distributed Bernoulli random variables with mean $\le cp$, the $X_{i, j}, j \in \{1, \dots, L\}$ are independent, and $L \le 2|D(x, r)|/(9r_k^2)$. The bound on $L$ uses that $r_k \le 10 r$. By a union bound and Hoeffding's inequality, we then have
	$$
	\P(|[\fD_k] \cap D(x, r)| \ge \frac{1}{3} |D(x, r)|) \le 9 r_k^2 \P\Big(\sum_{j=1}^L  X_{1, j} \ge \frac{|D(x, r)|}{27 r_k^2}\Big) \le 9 r_k^2 e^{- c r^2/r_k^2}.
	$$
	Taking $r_k = r^{1/2 + o(1)}$ and combining this with the bound in \eqref{E:fDk-estimate} then yields the result.
\end{proof}	

We now translate Theorem \ref{T:BssetI} to give two results about random SSPs. We start with a basic version for constant $\ka$ and $V = \Z^2$. For this version, we give concrete bounds.
	
	\begin{theorem}
		\label{T:random-red-blue}
Consider a random SSP $(\fR, \fB)$ on $\Z^2$ driven by potentially random clocks $X_\fR, X_\fB$, a collection of blue seeds $\fB_*$ and constant parameter $\ka > 4000$.
Suppose additionally that $\fB_*$ is stochastically dominated by an i.i.d.\ Bernoulli process $B'$ with parameter $p > 0$.		
	Let $\fC$ be defined from $\fB_*$ as in Theorem \ref{T:BssetI}. Then for small enough $p > 0$, the event where
	$$
	\fB_* = \bigcup_{i=1}^\infty A_i(\fB_*) \quad \text{ and } \quad [\fC] \text{ has only bounded components}
	$$
	is almost sure. Also for a universal $c > 0$, the probability of the event
	\begin{equation*}
	\begin{split}
\cE = \Big\{\fB_* = \bigcup_{i=1}^\infty A_i(\fB_*), [\fC] \text{ has only bounded components}, 0 \in [\fD]^c, [\fC]^c &\sset \fR(\infty)\Big\} 
\end{split}
	\end{equation*}
	is at least $1 - cp$. Moreover, on the event $\cE$, we have $D(0, t/(2\ka))) \sset [\fR(t)]$ for all large enough $t$. More precisely, for any $t_0 > 0$ we have
	\begin{equation}
	\label{E:D0kk}
	\begin{split}
	\P(\{D(0, t/(2\ka))) &\sset [\fR(t)] \text{ for all } t \ge t_0\} \; | \; \cE) \\
	&\ge 1 - \exp\lf(\log(cp) \exp (\log (t_0/\ka)/(c\log \log (t_0/\ka)))\rg).
	\end{split}
	\end{equation}
	\end{theorem}
	
	\begin{proof}
	First, for small enough $p$ the event $\fB_* = \bigcup_{i=1}^\infty A_i(\fB_*)$ is almost sure by Lemma \ref{L:Ak-estimate}. Second, the bound \eqref{E:Mn-BC} in Lemma \ref{L:largest-scale} ensures that $[\fD]$ (and hence $[\fC]$) has only bounded components almost surely. The second inequality in \eqref{E:Mrrk} ensures that $0 \in [\fD]^c$ with probability at least $1-cp$.
	On these events, $\fC^c \sset \fR(\infty)$ by Theorem \ref{T:BssetI} in the $V = \Z^2$ case.
	Next, observe that since all clocks take values between $0$ and $\ka$, that $D(0, t/\kappa) \sset \fR(t) \cup \fB(t)$ deterministically. Therefore the event 
	$$
	\cE \cap \{ D(0, t/(2\ka))) \not \sset [\fR(t)] \text{ for all } t \ge t_0 \}
	$$
	is contained in the event where there exists some $t > t_0$ for which a path in $\fC$ connects $D(0, t/2\kappa)$ and  $D(0, t/\kappa)$. Letting $M(0, r)$ be as in Lemma \ref{L:largest-scale}, this is in turn contained in the event
	$$
    \{\text{There exists } r \ge t_0/\ka, r \in \N : M(0, r) \ge r/2\}.
	$$
	We bound this again by \eqref{E:Mn-BC} in Lemma \ref{L:largest-scale}.
	\end{proof}

We also give a version for variable $\ka, V$.

\begin{theorem}
	\label{T:random-red-blue-var-ka}
	There exists a universal $p_0 > 0$ such that the following holds.
	Consider a random SSP $(\fR, \fB)$ on $\Z^2$ driven by potentially random clocks $X_\fR, X_\fB$ and a collection of blue seeds $\fB_*$. Suppose that the following conditions hold for some inner radius $r \ge 0$, outer radius $R \ge r$ and some $\ka_0 > 4000$.
	\begin{itemize}
		\item The parameter $\ka:E \to (0, \infty)$ satisfies $\ka > 4000$ pointwise and $\ka(u, v) \le \ka_0$ whenever $|u| > R$.
		\item Define $\fB_\clubsuit := \fB_* \cap D(0, r/2)^c$. Then $\fB_\clubsuit$ is stochastically dominated by an i.i.d.\ Bernoulli process with parameter $p \le p_0$. In the remainder of the theorem statement, the sets $\fD, \fC$ are defined with respect to the seed set $\fB_\clubsuit$.
		\item There exists a positive probability event $\cE$ such that on $\cE$, we can find a (potentially random) connected set $V \sset \Z^2$ with $D(0, r/2) \sset V^c \sset D(0, r)$ such that  $\del V \cap [\fD] = \emptyset$ and $\del V \sset \fR(\infty)$.
	\end{itemize}
Then conditionally on the event $\cE$, a.s. we have that
$$
[\fC]^c \cap V \sset \fR(\infty) \qquad \text{and all components of } \fB(\infty) \text{ are bounded}.
$$
Moreover, conditionally on $\cE$, a.s. we have that
$
D(0, t/(3 \ka_0)) \sset [\fR(t)]
$
for all large enough $t$.
\end{theorem}

\begin{proof}
Again, this follows from Theorem \ref{T:BssetI}. The first condition of Theorem \ref{T:BssetI} is guaranteed by assumption. For the second and third conditions, by the monotonicity established in Lemma \ref{L:dom-seeds}, it suffices to show that these statements hold for $\fB_\clubsuit$, rather than the smaller $\fB_\diamond := \fB_* \cap V$. The second statement then follows by the same reasoning as in Theorem \ref{T:random-red-blue}, and the third condition is guaranteed by the third assumption.

The boundedness of the components of $\fB(\infty)$ and the claim about $D(0, t/(3 \ka_0)) \sset [\fR(t)]$ uses the essentially the same reasoning as in Theorem \ref{T:random-red-blue}. The only difference is that since $\ka$ is not constant but only eventually bounded above by $\ka_0$, we can only guarantee that for any $\ep > 0$, 
$$
D(0, t(1- \ep)/\kappa_0) \sset [\fR(t) \cup \fB(t)]
$$
for all large enough $t$.
\end{proof}

\subsection{Interacting pairs}
\label{SS:interacting-pairs}

To prove the almost sure lower bound on the speed in Theorem \ref{thm:main}, we will need to analyze pairs of random SSPs defined at different scales. Suppose we have two potentially random SSPs $(\fR, \fB), (\fR', \fB')$ defined on two copies of $\Z^2$, labelled $\Z^2, \Z^{2 \prime}$. We define a map $P:\Z^2 \to \Z^{2 \prime}$ by setting $P(x, y) = (\floor{x/2}, \floor{y/2})$. We can think of this map as identifying $2 \times 2$ boxes in $\Z^2$ with a coarser lattice $\Z^{2 \prime}$.

Suppose that $(\fR, \fB), (\fR', \fB')$ interact within a disk of radius $r$ according to the following rule:
\begin{itemize}
	\item For $x \in D(0, r)$, if $P(x) \in \fR'(\infty)$, then either $x \in \fR(\infty)$, or else $x$ is a blue seed for $(\fR, \fB)$.
\end{itemize}
The idea is that we can use this rule to propagate survival of the process $\fR'$ to the process $\fR$.

With this rule, the following result is immediate.

\begin{lemma}
	\label{L:interacting-RB}
	In the setup of Theorem \ref{T:random-red-blue-var-ka}, suppose that the first two conditions hold. Suppose also that $(\fR, \fB)$ is coupled to another process $(\fR', \fB')$ so that on some event $\cE$, the interaction rule above holds, and there exists a connected (random) set $V$ such that 
	$$
	D(0, r/3) \sset V^c \sset D(0, r), \qquad P \del V \sset \fR'(\infty), \qquad \del V \cap [\fD] = \emptyset,
	$$ where $\fD$ is constructed as in Theorem \ref{T:random-red-blue-var-ka}. Then the final condition of Theorem \ref{T:random-red-blue-var-ka} also holds with the event $\cE$ and the random set $V$, and hence so does that theorem.
\end{lemma}

To apply Lemma \ref{L:interacting-RB} in conjunction with Theorem \ref{T:random-red-blue-var-ka}, we will need a few lemmas to help us find sets $V$. 

\begin{lemma}
		\label{L:two-seed-sets}
		In the setup above, let $\fD_k, \fD'_k$ be as defined at the beginning of Section \ref{S:inputs}. Then for all $k$, all components of the set $\fF_k := P \fD_k \cup \fD_k'$ have diameter at most $r_k$.
	\end{lemma}
	
	\begin{proof}
		The proof is similar to the proof of Lemma \ref{L:IJlemma}. We go through the inductive argument more briefly this time through. Again, we use induction on $k$. The definition of $A_1(\fB_\diamond)$ guarantees that any component of $\fF_1$ will intersect at most one component of $P\fD_1$ and one component of $\fD_1'$. Any component of $P\fD_1$ is a ball of radius $r_1/200$ and any component of $\fD_1'$ is a ball of radius $r_1/100$. Therefore components of $\fF_1$ have radius less than $r_1$.  Now suppose that the claim holds for $\fF_{k-1}$, and let $C$ be a connected component of $\fF_k$. 
		
		As in the proof of Lemma \ref{L:IJlemma}, we can use the inductive hypothesis to check that $C$ overlaps with at most one component $F'$ of $D(A_k(\fB_\diamond'), r_k/100)$ and at most one component $F$ of $P D(A_k(\fB_\diamond), r_k/100)$, each of which have diameter at most $r_k/10$. Attaching $F \cup F'$ to connected components of $\fF_{k-1}$ can only increase the total diameter by at most $2 r_{k-1}$, by the inductive hypothesis, so $\operatorname{diam}(C) \le r_k/10 + 2 r_{k-1} \le r_k$.
	\end{proof}
	
	Using Lemma \ref{L:two-seed-sets}, we can get an analogue of Lemma \ref{L:largest-scale} for the sets $\fF_k$. We only record the parts of that lemma that we will need moving forward.
	
	\begin{lemma}
		\label{L:random-two-seed-sets}
		Suppose that $\fB_\diamond, \fB_\diamond'$ are blue seed sets that are both stochastically dominated by i.i.d.\ Bernoulli processes of mean $p > 0$. Let $\fF := P \fD \cup \fD'$. Let $M(x, r)$  be the diameter of the largest component of $[\fF]$ that intersects $D(x, r)$. For some absolute $c > 0$, we have
			\begin{equation}
			\label{E:Mrrk'}
			\P\lf( M(x, r) > r_k \rg) \le r^2 (cp)^{2^{k-1}}.
			\end{equation}
			Moreover, letting $P_{a, b}$ denote the probability that there is a path in $[\fF]$ from $D(0, a)$ to $D(0, a + b)^c$, then for $a \ge c$ and $b \ge \exp((\log \log a)^2)$, we have 
			\begin{equation}
			\label{E:Pab-bound}
			P_{a,b} \le \exp \lf( \log (cp) \exp(\log b /(c \log \log b)\rg).
			\end{equation}
		\end{lemma}
	
	\begin{proof}
	The proof of \eqref{E:Mrrk'} is exactly the same as the proof of \eqref{E:Mrrk} in Lemma \ref{L:largest-scale}. For the bound on $P_{a, b}$, we use the bound \eqref{E:Mrrk'} on $M(0, a)$, the bound $r_k \le c \exp (2k \log k)$, and the fact that the event for $P_{a, b}$ implies $M(0, a) \ge b$.
	\end{proof}
	
\section{Linear growth in dimension $d = 2$}
In this section we define carefully constructed events in order to apply SSP to the SI model.

\label{S:linear}
\subsection{A block construction for the SI process}
\label{S:SI-colouring}

If we change the rates of the susceptible and infected  particles to $\theta D_S$ and $\theta D_I$ then it effectively speeds up the process by a time factor of $\theta$.  Thus by a suitable rescaling, we can without loss of generality assume that $D_S=1$ and let $\beta=D_I$.

Fix a side length $L = 2^k$ for some large $k \in \N$. We subdivide $\Z^2$ into a  collection of  dyadic blocks
\begin{equation}
\label{E:BL-construct}
\sB=\sB_L:=\Big\{zL +\big\{-L/2, \ldots ,L/2-1\big\}^2:z\in \Z^2\Big\}.
\end{equation}
This induces a natural map $f = f_L:\Z^2 \to \sB_L$.
Via a somewhat complicated construction, we will apply Theorems~\ref{T:random-red-blue} and \ref{T:random-red-blue-var-ka} to show that the infection grows linearly. Given a side length $L$, for some small $\alpha>0$, let $\xi = \xi_L = \alpha L^2$. This will be the minimum speed at which the colouring of blocks spreads for blocks far away from the origin. 

We will now define a \emph{colouring process} for blocks $B \in \sB_L$. The colouring process will depend on a slow initial speed parameter $\chi \ge 8 \xi$, a small radius parameter $r \ge 0$ and a large radius parameter $R \ge 2 r \chi/\xi$.

\begin{definition}
	\label{D:SI-colouring}
We will let $\tau_B$ denote the (stopping) time when a block $B$ is first coloured. For a vertex $u\in B$ we will use $\tau_u$ as shorthand for $\tau_B$. A block $B$ becomes coloured at time $t$ if one of the following events happens:
\begin{enumerate}[label=(\alph*)]
\item A neighbouring block $B'$ was coloured at time $\tau_{B'} = t - \chi$ and $d(B, 0) \le r$, \textbf{or} a neighbouring block $B'$ was coloured at time $\tau_{B'} = t - 8\xi$ and $R \ge d(B, 0) > r$, \textbf{or} a neighbouring block $B'$ was coloured at time $\tau_{B'} = t - \xi$ and $d(B, 0) > R$.
\item An infected particle enters $B$ for the first time at time $t$.
\item A block $B'$ becomes coloured according to rule (b) at time~$t$ and the infected particle that enters $B'$ is within distance  $\alpha L$ of~$B$.
\end{enumerate}
\end{definition}
In case (c) we call this a multi-colouring, where up to three blocks may be coloured simultaneously. 
In case (b) we call the infected particle that entered $B$ the \emph{ignition particle} of $B$.  In case (c) the ignition particle is the ignition particle from block $B'$. In cases (b) and (c) we say that the block $B$ is \emph{ignited} at the location $x$ where an infected particle first entered $B$ (for case (b)) or first entered the relevant neighbour of $B$ (for case (c)). All ignition locations are within distance $\al L$ of $B$.

Moving forward, we say that a particle $a$ is \emph{coloured} the first time that it is in a coloured block, either because the block was coloured or because it entered a coloured block. We let $X_t^*$ denote the process of coloured particles up to time $t$ and let $\cF_t^*$ denote the $\sigma$-algebra generated by $X_t^*$. We let $\iota_a$ be the stopping time when particle $a$ becomes infected.

A key element of our analysis is to regard the collection of random walks as given by a Poisson process on $\fW$, the space of cadlag sample paths $w(t):\R\to\Z^2$ such that 
\begin{equation}\label{eq:sublinearGrowth}
\frac1{t}|w(t)|\to 0
\end{equation}
as $|t| \to \infty$.  Elements of $\fW$ will represent the trajectories of particles.
Let $\sW_u$ denote the measure on $\fW$ given by a continuous time random walk on $\Z^2$ over all time $t\in \R$ that is at $u\in \Z^2$ at time $0$. Note that a simple random walk satisfies~\eqref{eq:sublinearGrowth} almost surely.  Furthermore, let $\sW = \sum_{u\in \Z^2} \sW_u$.  We let $\sP$ denote a Poisson process on $\fW$ with intensity measure $\mu\sW$.  If we remove the initial infected particle from the origin, we can interpret all remaining particles and their trajectories as being given by a sample from $\sP$.  

In order to obtain spatial independence of various events we will give an alternative construction of the SI process $X_t$ with the same law based on a collection of Poisson processes $\sP_B, B \in \sB$ which are IID and equal in distribution to $\sP$.  We will use these processes to construct $X_t$ as follows.  In Section~\ref{s:coupling} we will prove that this gives a valid coupling.

Let $\cM_B$ denote the $\sigma$-algebra generated by 
\begin{enumerate}[label=(\roman*)]
	\item the independent Poisson process $\sP_B$,
	\item a walk $W_{B,\operatorname{ig}}$ sampled independently from $\sW_0$, which will encode the trajectory of an ignition particle. 
\end{enumerate}
For each block $B$ the processes $\cM_B$ are IID.  We call a particle $a$ in $\sP_B$ \emph{simple} and let $a(t)$ denote its trajectory.

We will build $X_t$ according to the $\cM_B$ in such a way that particles coloured in $B$ correspond to particles in $\sP_B$ but with a time shift of length $\tau_B$. We will abuse notation somewhat by conflating a particle $a$ in some $\sP_B$ with a particle in $X_t$ matched according to our construction. Defining
\begin{equation}
\label{E:atti}
\oa(t,t',i):=\begin{cases} a(t-t') &t\leq i\\
a(i-t'+\beta(t-i)) &t>i,
\end{cases}
\end{equation}
the trajectory of $a$ in $X_t$ will be given by $\oa(t,\tau_B,\iota_a)$, which we shorten to $\oa(t)$ when clear. This definition incorporates the change in speed after the infection time $\iota_a$. This means that prior to infection, a particle's location in $X_t$ at time $t$ corresponds to its location at time $t-\tau_B$ in $\sP_B$.  Thus, the particle locations in $\sP_B$ at time $0$ tell us the particle locations in $X_t$ at time $\tau_B$.

Let $B_0$ be the block containing the origin so that $\tau_{B_0}=0$.  At time $0$ we add all simple $\sP_{B_0}$-particles that are in $B_0$ at time $0$ to $X_t^*$ plus an infected ignition particle at the origin. The simple particles evolve according to their paths $a(t)$ while the initially infected particle evolves according to $W_{B_0, \operatorname{ig}}(t)$. Particles become infected if they enter the same vertex as another infected particle after which they move along the path $a(t)$ at rate $\beta$, as in \eqref{E:atti}.  New particles can enter the coloured collection of particles $X_t^*$ in two ways, either when a new block becomes coloured or when an uncoloured particle enters a coloured block for the first time.

{\bf Case 1:  A newly coloured block:} When a block $B$ is coloured for the first time according to Definition \ref{D:SI-colouring} we add  particles to $X_t^*$ as follows.  If a particle $a$ from $\sP_B$ is in $B$ at time $0$, then we add it to $X_t^*$ at time $t=\tau_B$ if for all $0\leq t < \tau_B$  we have that $\tau_{a(t-\tau_B)} > t$.  This condition is equivalent to saying that a particle with trajectory $a(t-\tau_B)$ first hits a coloured block at time $\tau_B$.  For $t\geq \tau_B$ the future trajectory of the particle is given by $\oa(t)$.

If $B$ is ignited at a vertex $x \in \del B$, this ignition particle follows special rules.  Instead of continuing to follow the trajectory given by the block it was initially coloured by, when it becomes the ignition particle of $B$ at time $\tau_B$ we alter its future trajectory to $x+ W_{B,\operatorname{ig}}(\beta(t-\tau_B))$.

{\bf Case 2:  Particle first entering a coloured block:} For times $t\in(\tau_B,\tau_B+\chi]$ new particles are revealed in $B$ in the process $X_t^*$ according to the following rules.  If $a$ is a particle in $\sP_B$ that enters $B$ at time $t-\tau_B$, then we add it to $X_t^*$ at time $t$ if for all $0\leq s < t$  we have that $\tau_{a(s-\tau_B)} > t$.  This condition is equivalent to saying that a particle with trajectory $a(s-\tau_B)$ first hits a coloured block at time $t$.  For $s\geq t$ the future trajectory of the particle is given by $\oa(s)$.  Particles can only join $X^*_t$ in this way during the time interval $(\tau_B,\tau_B+\chi]$ since at time $\tau_B+\chi$ the neighbouring blocks of $B$ are already coloured by construction.


Let $H_B$ be the set of all the particles that are first coloured in $B$. For blocks close to the origin, our colouring rules do not guarantee that $H_B$ has to contain many particles. However, at distances at least $r\chi/\xi$ from the origin, rule (a) guarantees that our colouring process is moving at a linear speed depending only on $L$, which helps to ensure that there are many particles in $H_B$. With this in mind, define $H^-_B$ to be all particles $a\in \sP_B$ that satisfy the following additional constraint:
\begin{equation}
\label{eq:principal}
\sup_{t\leq 0} d(a(t),B) - \frac{L}{20} \lfloor |t\xi^{-1}| \rfloor = 0.
\end{equation}
We call the particles in $H_B^-$ \emph{principal particles}. 

\begin{lemma}
\label{L:principal-particles}
For $d(0, B) \ge r\chi/\xi$, $H_B^- \subset H_B$.
\end{lemma} 

\begin{proof}
A principal particle $a\in \sP_B$ must have been present in $B$ at time 0.  Furthermore, suppose that $\tau_{a(s-\tau_B)} \leq s$ for some $0\leq s <\tau_B$.  Then if $a(s - \tau_B )\in B'$ then \begin{equation}
\label{E:dBB'L}
d(B',B)\leq \frac{L}{20} \lfloor |s - \tau_B|\xi^{-1}\rfloor
\end{equation}
by equation~\eqref{eq:principal}. Moreover, the colouring rule in Definition \ref{D:SI-colouring}(a) and the assumption that $d(0, B) \ge r\chi/\xi$ guarantees that 
$$
\tau_B \le \frac{\chi r}{L} + \frac{8 \xi d(0, B)}{L} \le \frac{9 \xi d(0, B)}{L},
$$
so we further have that $d(B',B)\leq d(B, 0)/2$. In particular, since $d(0, B) \ge r \chi / \xi \ge 8r$, this implies that the shortest path of blocks from $B'$ to $B$ does not hit any block within distance $r$ of the origin. Along this path, all colouring times of adjacent blocks are spaced apart by at most $8\xi$. Moreover, \eqref{E:dBB'L} implies that $B, B'$ are at most $\frac{1}{20} \lfloor |s - \tau_B|\xi^{-1}\rfloor$ blocks apart, and so $\tau_{B'} \ge \tau_B - \frac{8\xi}{20} \lfloor |s - \tau_B|\xi^{-1}\rfloor > s$, which is a contradiction.
\end{proof}

\subsection{Blue Seeds}
\label{S:blue-seeds}

In this section we give a set of conditions on $\cM_B$ which will ensure the efficient spread of the infection to the neighbouring blocks if the block $B$ is ignited, i.e. coloured according to rule (b) or (c).  The complement of this event will correspond to marking the block as a blue seed in an associated SSP. For this section and the remainder of Section \ref{S:linear}, we fix an arbitrary $\ka_0 > 4000$.

Setting some notation, let $U_x$ denote the union of all blocks in $\sB_L$ within distance $\alpha L$ of $x$. This includes the block containing $x$. When the meaning is clear from context, we will also let $U_x$ denote the set of vertices contained in a block in $U_x$.

 Furthermore, if $x$ is the ignition site for the block $B$, all blocks in $U_x$ are coloured at or before time $\tau_B$. Set 
\[
\partial B^\#= \bigcup_{B'\in \sB_L}\{x\in \partial B':d(x,B) \leq \alpha L\},
\]
which is the set of possible initial locations for the ignition particle of $B$ (the shape of the set $\partial B^\#$ resembles a hashtag $\#$).
We now give rules for when a block is a blue seed.

The ignition particle for $B$ is infected, starts at $x\in\partial B^\#$ and follows the path $x+ W_{B_x,\operatorname{ig}}(\beta(t-\tau_B))$, where $B_x$ is the block containing $x$. Define the event
\[
\cA^{(1)}_{B}=\bigcap_{B':d(B,B')\leq 2}\Big\{\sup_{0\leq s\leq \xi/ \log \log L} |W_{B',\operatorname{ig}}(\beta s)| \leq \frac12 \alpha L\Big\}.
\]
Since $d(B_x, B) \le 2$ for all $x \in \del B^\#$, this event ensures that the ignition particle for $B$ remains in $U_x$ until time $\tau_B + \xi/ \log \log L$, and hence cannot become the ignition particle of another block prior to this time.

For each lower scale block $B'\in \sB_{L/2}$ with $d(B, B') \le 1$ and $x \in \del B^\#$ define the event
\begin{align*}
\cA^{(2)}_{B,B',x}=\bigcup_{a,s,s'}\bigg\{ a(s) = x+ W_{B_x,\operatorname{ig}}(\beta s), a(s')\in B', \forall 0\leq s''<s', a(s'')\in U_x\bigg\},
\end{align*}
where the union is over $a \in H_B^- \cap H_B$ and $0\leq s \leq s' < \tfrac{\xi}{\log\log L}$.
This event asks for at least one principal particle in $B$ to have a rate 1 trajectory that intersects the ignition particle's trajectory at some time $s$, to stay within $U_x$ until some time $s'$, at which point it enters $B'$.  We will see that this guarantees that $B'$ is ignited before time $\tau_B + \tfrac{\xi}{\log\log L}$ if it has not been otherwise coloured.

\begin{lemma}
	\label{L:lower-scale-ignition}
	Suppose that a block $B \in \sB_L$ is ignited at time $\tau_B$ at location $x$ in a colouring process with arbitrary parameters $\chi, r, R$, and that $\cA^{(1)}_{B}\cap\cA^{(2)}_{B,B',x}$ holds. Then in any colouring process at scale $L/2$, the finer scale block $B'$ is coloured by time $\tau_B + \xi/(2\kappa_0)$.
\end{lemma}

\begin{proof}
Select some $a\in H_B^- \cap H_B$ and  $s,s'$ that make the event $\cA^{(2)}_{B,B',x}$ hold. The particle $a$ must be infected by time $\iota_a \in [\tau_B, \tau_B+s]$. Indeed, if particle $a$ does not get infected before time $\tau_B + s$ then at this time it will meet the ignition particle and become infected. During the interval $[\tau_B,\tau_B + \iota_a + \beta^{-1}(s'-\iota_a))$, particle $a$ remains in $U_x$ and so does not become the ignition particle of another block in $\sB_L$.  Then at time $\tau_B + \iota_a + \beta^{-1}(s'-\iota_a)$, which satisfies the inequality
	\[
	\tau_B + \iota_a + \beta^{-1}(s'-\iota_a)) \leq \tau_B + s'(1\vee\beta^{-1}) < \tau_B + \frac{\xi}{2\kappa_0},
	\]
	particle $a$ is in the finer scale block $B'$.  If $B'$ has not already been coloured at this time in the finer scale process, then at this time it is ignited.
\end{proof}

We define
\begin{equation}
\label{E:cA123}
\cA_B=\cA_B^{(1)}\cap\bigcap_{x\in\partial B^\#} \bigcap_{B'\in \sB_{L/2}:d(B,B')\le 1}\cA^{(2)}_{B,B',x}
\end{equation}
We call the block $B$ a \textbf{blue seed} if $(\cA_B)^c$ holds.  When $\cA_B$ holds, if $B$ is ignited then all its neighbouring blocks in both $\sB_L$ and $\sB_{L/2}$ will be coloured before time $\tau_B + \xi/\kappa_0$.

\begin{lemma}\label{L:blue-seed-indep}
	For $d(0, B) \ge r \chi/\xi$, the event $\cA_B$ is measurable given the $\sig$-algebras $\cM_{B'}, B' \in \cB_L, d(B, B') \le 2$.
\end{lemma}

\begin{proof}
	From the definitions, the only potential dependence on information not contained in $\cM_{B'}, B' \in \sB_L, d(B, B') \le 2$ is in the set $H_B^- \cap H_B$ used in the definition of $\cA^{(2)}_{B,B',x}$. Since $d(0, B) \ge r \chi/\xi$, Lemma \ref{L:principal-particles} implies that $H_B^- \cap H_B = H_B^-$, which is $\cM_B$-measurable.
\end{proof}

\begin{proposition}\label{p:blueSeed}
	There exists $\al > 0$ such that for any $\epsilon>0$ there exists $L_\epsilon$ such that the following holds. For any colouring process defined at a scale $L \ge L_\epsilon$ and any valid parameter choices $\chi, r, R$, the probability that a block $B$ is a blue seed when $d(0, B) \ge r \chi/\xi$ is bounded above by $\ep$:
	\[
	\P[(\cA_B)^c] \leq \epsilon.
	\]
\end{proposition}

Proposition \ref{p:blueSeed} will be proven in Sections \ref{S:tools} and \ref{S:prop35}. An immediate consequence of the two results above is the following.

\begin{corollary}
	\label{C:stoch-dom}
	There exists $\al > 0$ such that for any $\epsilon>0$ there exists $L_\epsilon$ such that the following holds. For all $L > L_\ep$ and any valid parameter choices $\chi, r, R$, letting $\fB_* \sset \sB_L$ be the set of blue seeds associated to these parameters, the set
	$$
	\fB_\clubsuit := \fB_* \cap \{B\in \fB_L : d(0, B) \ge r \chi/\xi\}
	$$
	is stochastically dominated by an i.i.d. Bernoulli process of intensity $\ep$.
\end{corollary} 

\begin{proof}
We appeal to Theorem 0.0(i)\footnote{It really is Theorem 0.0, this is not a typo} in \cite{liggett1997domination}, which states the following. Let $d \ge 1$, and suppose that $X:\Z^d \to \{0, 1\}$ is a random process such that for any vertex $v \in \Z^d$, the conditional probability that $X(v) = 1$ given all the values of $X$ on vertices at $\ell^\infty$-distance at least $k$ away from $v$ is at most $\ep$. Then $X$ is stochastically by an i.i.d.\ Bernoulli process $Y$ on $\Z^d$ such that $\E Y(0) = f(\ep)$, where $f(\ep) \to 0$ with $\ep$. Here the function $f$ depends on $k$ and $d$. 
In our setting, we use that $\sB_L$ has a graph structure which is (isomorphic to) $\Z^2$ and let $X = \fB_\clubsuit$. Lemma \ref{L:blue-seed-indep} and the independence of the different $\sig$-algebras $\cM_B$ gives the estimate for finite-range dependence, and Proposition \ref{p:blueSeed} gives the estimate on the individual seed probabilities. Note that the $L_\ep$ we get in Corollary \ref{C:stoch-dom} may be larger than the $L_\ep$ in Proposition \ref{p:blueSeed}.
\end{proof}

\subsection{Preliminary lemmas}
\label{S:tools}
In this short section we prove a few simple estimates about Poisson processes and random walks that we will make use of.
With $W_t$ a continuous time random walk, for a vertex $x$ in a block $B \in \sB_L$ define
\[
p_x = \P\Big[\forall t\geq 0: d(x+W_t,B) \leq \frac{L}{20} \lfloor t\xi^{-1} \rfloor \Big],
\]
where $\xi = \al L^2$.
At time $\tau_B$ the principal particles $H_B^-$ are Poisson distributed with intensity $p_x \mu$.
For any $x \in B$ such that $d(x,B^c) \geq \frac{L}{4}$ we have that
\begin{align}\label{eq:principalLB}
p_x &\geq 1 - \P\Big[\sup_{0\leq t \leq \alpha L^2}  |W_t| > \frac{L}{4} \Big] - \sum_{j\geq 1} \P\Big[\sup_{0\leq t \leq (j+1)\alpha L^2}  |W_t| > \frac{Lj}{20} \Big]   \geq \frac34,
\end{align}
for all large enough $L$ and $\alpha$ sufficiently small.

\begin{lemma}\label{l:HBsize}
For $\al$ sufficiently small, for large enough $L$ the number of particles in $H_B$ is stochastically dominated by a Poisson random variable with mean $\frac32 \mu L^2$.
\end{lemma}

\begin{proof}
Let $H_B^+$ be all particles in $\sP_B$ that are in $B$ at some point during the time time interval $[0,\al L^2]$ so $H_B\subset H_B^+$.  The size of $|H_B^+|$ is Poisson distributed with mean,
\[
\E |H_B^+| = \sum_{x\in \Z^2} \mu  \P[\exists t\in[0,\al L^2]: W_t+x \in B ] \leq  \frac32 \mu L^2,
\]
where the inequality uses that $L$ is sufficiently large.
\end{proof}



Next, let $H$ be a set of particles performing independent rate-$1$ random walks. 
The following lemma will be useful for counting how many particles in $H$ hit a deterministic trajectory $x:[0,T]\to\Z^2$. We will ask that all $H$-particles remain in a set $U\subset \Z^2$ until they hit $x(t)$. With this in mind, we write 
$$
S(a,U)=\inf\{s\geq 0: a(s)\not\in U\},
$$
and define
\[
N(H,T,U,x):= \sum_{a\in H} I(\{t\in [0,T]: a(t) = x(t), S(a,U) > t\}\neq \emptyset).
\]
This counts the number of particles in $H$ that intersect the path $x$ in the interval $[0, T]$ prior to leaving the set $U$. We also let
\[
R(H,T,U,x):= \sum_{a\in H} \int_0^T I(a(t) = x(t), S(a,U) > t)dt
\]
be the aggregate intersection time of the particles with $x(t)$ prior to leaving $U$.
\begin{lemma}
\label{l:hittingNumberSimple}
There exist constants $C_1,C_2>0$ such that for $T\geq 2$, any set of particles $H$ and any path $x:[0, T] \to \Z^2$ and any set $U$ we have
\[
\P\bigg[N \leq \frac{C_1}{\log T} \E[R ]\bigg] \leq 2\exp\bigg(-\frac{C_2}{\log T} \E[R ]\bigg)
\]
where $N=N(H,T,U,x)$ and $R=R(H,T,U,x)$.
\end{lemma}

\begin{proof}
For a particle $a \in H$, define
\begin{align*}
    n_a &= \P\bigg[\int_0^T I(a(t) = x(t), S(a,U) > t)dt>0\bigg],\\
    r_a &= \E\bigg[\int_0^T I(a(t) = x(t), S(a,U) > t)dt\bigg]
\end{align*}  
Setting $\varsigma_a = \inf\{t\geq 0: a(t) = x(t), S(a, U) > t\}$ and letting $P^t_{y, y'} = \P(a(t) = y' \;|\; a(0) = y)$ be the transition probability for a rate $1$ random walk we have that
\begin{align*}
r_a &\leq \E[I(\varsigma_a \leq T) \int_{\varsigma_a}^T P^{t-\varsigma_a}_{x(\varsigma_a),x(t)} dt ] \\
&\leq \E[I(\varsigma_a \leq T) \int_{\varsigma_a}^T \frac{C}{1+t-\varsigma_a} dt ]\\
&\leq C\log (T) \E[I(\varsigma_a \leq T) ] =  C\log (T)n_a.
\end{align*}
Hence
\[
\E[N] \geq \frac{1}{C\log T} \E[R] 
\]
and the result follows by standard concentration bounds since $N$ is a sum of indicators.
\end{proof}

\subsection{Proof of Proposition~\ref{p:blueSeed}}
\label{S:prop35}

Fix a small value of $\al > 0$, let $\ep > 0$, and let $B$ be such that $d(0, B) \ge r \chi/\xi$.
First, recalling that $\xi = \al L^2$, by standard random walk estimates, we have that
\[
\P[\sup_{0\leq s\leq \xi/\log \log L} |W_{B,\operatorname{ig}}(\beta s)| \le \frac{1}{2} \al L ] \to 1
\]
as $L \to \infty$.
Therefore by a union bound over $B$ and its 8 neighbouring blocks we have that for large enough $L$,
\begin{equation}
\label{E:pA1B}
\P[\cA^{(1)}_{B}]\geq 1 - \epsilon/2.  
\end{equation}
On this event the ignition particle stays inside $U_x$ and is at least distance $\frac12 \alpha L$ away from the boundary of $U_x$ up to time $\xi$.

Now, since $d(0, B) \ge r \chi/\xi$, Lemma \ref{L:principal-particles} implies $H_B^- = H_B^- \cap H_B$.
Let $H^* \sset H_B^-$ be the set of principle particles in $B$ such that $d(a(0), B^c) \ge \frac14 L$. By~equation~\eqref{eq:principalLB} these have density at least $\frac34\mu$ and so $|H^*|$ stochastically dominates a Poisson random variable with mean $\frac16 \mu L^2$.  Defining the event $\cI_B^{(1)}=\{|H^*|\geq \frac1{10}\mu L^2\}$ we have that
\begin{equation}
\label{E:pIB}
\P[\cI_B^{(1)}] \geq 1-\exp(-cL^2)
\end{equation}
for some $c > 0$ depending only on $\mu$. This follows from a standard estimate on the concentration of a Poisson random variable. Next, let 
\begin{align*}
N_{B,x}&=N(H^*,\tfrac{\xi}{2\log \log L},U_x,x+W_{B,\operatorname{ig}}(\beta s)),\\
\qquad R_{B,x}&=R(H^*,\tfrac{\xi}{2\log \log L},U_x,x+W_{B,\operatorname{ig}}(\beta s)).
\end{align*}
By standard random walk estimates there exists $\phi_\alpha>0$ independent of $L$ such that
\[
\inf_{s\in[\tfrac{\xi}{3\log \log L}, \tfrac{\xi}{2\log \log L}]}\inf_{a\in H^*}\inf_{\substack{v\in U_x\\d(v,U_x^c)\geq \frac12 \alpha L}}\P[a(s)=v,S(a,U_x)>s] \geq \frac{1}{L^2 (\log L)^{\phi_\al}}.
\]
Since $d(x+W_{B,\operatorname{ig}}(\be s), U_x^c) \ge \frac12 \alpha L$ for all $s \le \xi/\log \log L$, this implies that 
\[
\E[R_{B,x}\mid H^*,W_{B,\operatorname{ig}},\cA^{(1)}_{B}] \geq (\log L)^{-\phi_\al} |H^*| \frac{\xi}{6\log \log L} \geq  \frac{cL^2}{(\log L)^{\phi_\al + 1}}.
\]
for some $c > 0$.
Therefore by Lemma~\ref{l:hittingNumberSimple}, for some new $c>0$ we have
\begin{equation}
\label{E:Nbx-bd}
\P[N_{B,x} \geq cL^2/(\log L)^{\phi_\al + 2} \mid \cA^{(1)}_{B}]  \geq 1-\exp(-c L^2/(\log L)^{\phi_\al + 2}).
\end{equation}
Let $H'\subset H^*$ be the particles counted by $N_{B,x}$. Now, given that $N_{B, x}$ is large, we want to find the probability of the event $\cA^{(2)}_{B,B',x}$.
Again using standard random walk estimates we can find a constant $\theta_\alpha>0$ independent of $L, B,B'$ and $x$ such that for any starting point $u\in U_x$ such that $d(u,\partial U_x)\geq \frac12 \alpha L$ the probability that a random walk started at $u$ enters $B'$ at a time before $\tfrac{\xi}{\log \log L}$ without exiting $U_x$ is at least $(\log L)^{-\theta_\alpha}$.  Therefore using \eqref{E:Nbx-bd} we have
\begin{equation}
\label{E:cA-estimate}
\begin{split}
&\P[\cA^{(2)}_{B,B',x}\mid \cA^{(1)}_{B}] \\
&\geq \P[\hbox{Bin}( cL^2/(\log L)^{\phi_\al + 2},(\log L)^{-\theta_\alpha})>1]-\exp(-c L^2/(\log L)^{\phi_\al + 2})\\
&\geq 1  -\exp(-c' L^2/(\log L)^{\phi_\al + \theta_\alpha}) -\exp(-c L^2/(\log L)^{\phi_\al + 2}).
\end{split}
\end{equation}
Taking a union bound over $B'$ and $x$ and gathering the estimates \eqref{E:pA1B}, \eqref{E:pIB}, and \eqref{E:cA-estimate} we have that
\[
\P[\cA_{B}]\geq 1-\epsilon,
\]
for large enough $L$, establishing Proposition~\ref{p:blueSeed}.

\subsection{Linear growth}
\label{SS:linear-growth}
Our next aim is to apply Corollary \ref{C:stoch-dom} in conjunction with Theorems \ref{T:random-red-blue} and~\ref{T:random-red-blue-var-ka} to ensure linear growth of the SI process.

The first step is to check that the construction above fits into the framework of Section \ref{S:SSP}. To put processes in the framework of that section, it will be more natural to use the time scaling of Section \ref{S:SI-colouring}, rather than the scaling from Section \ref{S:SSP}, which forces red clocks to be defined on the interval $[0, 1]$ rather than an arbitrary interval $[0, s]$. This clearly does not affect any theorem statements from Section \ref{S:SSP}. The parameter $\ka$ will still play the same role, so that blue clocks will take values in $[0, \ka(u, v) s]$.

Let $X_t$ be an SI process, and consider parameters $L, \chi, r, R$ and a colouring process for these parameters coupled to $X_t$ as in Section \ref{S:SI-colouring}. 
Let $f =f_L:\Z^2 \to \sB_L$ denote the natural correspondence \eqref{E:BL-construct}. We will define an SSP on $\Z^2$ based on Definition \ref{D:SI-colouring}, which gives a colouring process on $\sB_L$. Rules (b) and (c), which ignite blocks with infected particles, will cause the red process to spread. Rule (a) will cause the blue process to spread. The blue seed set $\fB_* \sset \sB_L$ defined via \eqref{E:cA123} corresponds to a blue seed set $\fB_* \sset \Z^2$ by the correspondence $f$. Here we abuse notation slightly by using $\fB_*$ for both sets.

The speed of spread from rule (a) is different depending on the location of the blocks. This results in a variable parameter $\ka:E \to [0, \infty)$. More precisely, with $\ka_0 > 4000$ as in Section \ref{S:blue-seeds}, define the process parameter $\ka:[0, \infty) \to E$ by
\begin{equation}
\ka(u,v) = \begin{cases}
\ka_0, \qquad &d(f(v), 0) > R,\\
8 \ka_0, \qquad &d(f(v), 0) \in (r, R],\\
\chi \ka_0/\xi, \qquad &d(f(v), 0) \le r.
\end{cases}
\end{equation}
Next, before defining the red and blue clocks, we need a systematic way of dealing with the fact that certain blocks may be ignited simultaneously because of rule (c). We work with this by setting certain edge clocks equal to $0$.
First, we let $u<_{\operatorname{SI}} v$  for two vertices $u, v \in \Z^2$ if both $f(u)$ and $f(v)$ are ignited by the same ignition particle at location $x$, and $d(x, f(u)) < d(x, f(v))$. Our ignition rules guarantee that either $d(x, f(u)) < d(x, f(v))$ or $d(x, f(u)) > d(x, f(v))$.
The random directed graph on $\Z^2$ with edges given by pairs $(u, v) \in \Z^2 \times \Z^2$ with $u <_{\operatorname{SI}} v$ is acyclic. 

For each directed edge $(u, v)$ between adjacent vertices in $\Z^2$, define red and blue clocks $X_\fR:E \to [0, \xi/\ka_0]$ and $X_\fB:E \to [0, \infty)$ as follows:
\begin{align*}
X_\fR(u, v) &= \begin{cases}
\lf(\tau_{f(v)} - \tau_{f(u)} \rg) \wedge \frac{\xi}{\ka_0}, \qquad &\tau_{f(v)} - \tau_{f(u)} > 0,
\\
0, \qquad &\tau_{f(v)} = \tau_{f(u)} \text{ and } u <_{\operatorname{SI}} v, \\
\frac{\xi}{\ka_0} \qquad &\text{ else.}
\end{cases} \\
X_\fB(u, v) &= \begin{cases}
\tau_{f(v)} - \tau_{f(u)}, \qquad \qquad \;\;\; \; &\tau_{f(v)} - \tau_{f(u)} > 0,
\\
0, \qquad &\tau_{f(v)} = \tau_{f(u)} \text{ and } u <_{\operatorname{SI}} v, \\
\frac{\xi \ka(u, v)}{\ka_0} \qquad &\text{ else.}
\end{cases}
\end{align*}
The data above defines an SSP, where all red sites correspond to ignited boxes in $\sB_L$.
\begin{prop}
	\label{P:SI-to-BR}
The clocks above along with the set of blue seeds $\fB_*$ a.s.\ define a finite speed SSP on $\Z^2$ with time changed by a factor of $\xi/\ka_0$. Moreover, a.s.\ for every $u \in \Z^2$, the colouring time $T(u)$ in this SSP equals $\tau_{f(u)}$. That is,
\begin{equation}
\label{E:BR-tau-eqn}
\mathfrak{B}(t) \cup \mathfrak{R} (t) = \{u \in \Z^2 : \tau_{f(u)} \le t \}.
\end{equation}
Finally, let $\operatorname{IGN} : =\{v \in \Z^2: f(v) \text{ is ignited or } v = 0\}$. Then a.s.\ 
\begin{equation}
\label{E:fRinfty}
\operatorname{IGN}^c \sset \{v \in \Z^2 : v \in \fB(\infty) \text{ and } X_\fB(u, v) \text{ fired to colour } v \text{ for some } u \in \Z^2\}.
\end{equation}
In particular, $
\fR(\infty) \sset \operatorname{IGN} \smin \fB_*.$
\end{prop}

The proof of Proposition \ref{P:SI-to-BR} is a straightforward check that all the definitions match up. It is somewhat lengthy, so we leave it to Section \ref{S:SI-to-BR}.

Given Proposition \ref{P:SI-to-BR} and Corollary \ref{C:stoch-dom}, our next aim is to apply Theorems \ref{T:random-red-blue} and \ref{T:random-red-blue-var-ka}. At this point, we will need to consider interactions between different SSPs, see Figure \ref{fig:second-argument} for the basic idea. Let $(\fR, \fB)[L, \chi, r, R]$ denote the SSP in Proposition \ref{P:SI-to-BR} with parameters $L, \chi, r, R$. For the remainder of this section, all SSPs will be coupled \emph{to the same} SI process $X_t$.

\begin{figure}
	\centering
		\includegraphics[width=4.7cm]{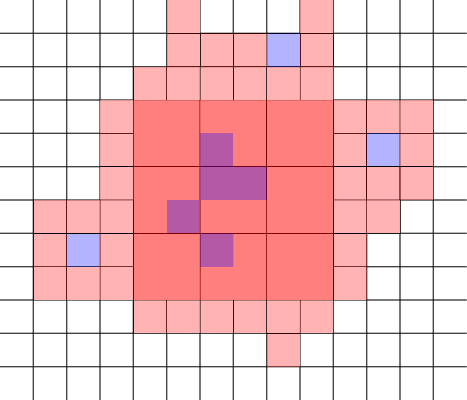}
	
\caption{We want to show that some SSP associated to $X_t$ at a fixed scale $L$ survives forever. We first find a scale large enough so that an associated SSP at that scale survives. We can do this since blue seed density decreases to $0$ with the scale by Proposition \ref{p:blueSeed}. We then look at the interaction between this coarse initial scale and a finer scale. By forcing the fine scale process to move slowly initially, we can guarantee that the red process at the fine scale contains the red process at the coarser scale, up to some fine scale blue seeds. This gives the fine scale process a good start even though there may be fine scale blue seeds near $0$. This is illustrated in the figure, where dark red represents the coarser scale, and light red and blue the finer scale. This allows the fine scale process to survive forever and eventually move faster than the coarse scale process. We can keep dropping down in scale until we reach the desired fixed scale.}
	\label{fig:second-argument}
\end{figure}

\begin{lemma}
\label{L:adjacent-boxes}
Consider fine and coarse processes
 $$
 (\fR_1, \fB_1) := (\fR, \fB)[L, 8\chi, r_1, R_1], \qquad (\fR_2, \fB_2):= (\fR, \fB)[2L, \chi, r_2, R_2].
 $$
 Suppose that $R_2 + 2L \le r_1$, and let 
 $$
 \tau = \min \{t \in [0, \infty) : f_L(\fR_1(t) \cup \fB_1(t)) \not \sset D(0, R_1) \}.
 $$
 In words, $\tau$ is the (stopping) time when the the fine colouring process exits the disk $D(0, R_1)$.
 
 Think of $(\fR_i, \fB_i)$ as being defined on a copy $\Z^2_i$ of $\Z^2$ and let $P:\Z^2_1 \to \Z^2_2$ be given by $P(x, y) = (\floor{x/2}, \floor{y/2})$. Then for $u \in \Z^2$, if $Pu \in \fR_2(t)$ and $t + \xi_{2L}/\ka_0 \le \tau$, then either $u \in \fR_1(t +\xi_{2L}/\ka_0)$, or else $u$ is a blue seed for $(\fR_1, \fB_1)$.
\end{lemma}

In the proof, we let $T_i:\Z^2_i \to [0, \infty)$ denote the colouring time map for $(\fR_i, \fB_i)$, let $\ka_i$ denote the parameter, and let $\fB_{i, *}$ denote the blue seed set.
\begin{proof}
Let $t$ be the first time when
\begin{equation}
\label{E:P-contained}
P(\fR_1(t) \cup \fB_1(t)) \not \sset \fR_2(t) \cup \fB_2(t).	
\end{equation}
We first check that $t \ge \tau$. First observe that if $u \in \Z^2$ is such that $f_L(u)$ was \emph{ignited} at some time $s \le t$, then $f_{2L}(Pu)$ would have also been ignited at time $s$, if it was not coloured prior to that time. Therefore the containment in \eqref{E:P-contained} can only be broken if a vertex $u$ was added to $\fR_1(t) \cup \fB_1(t)$ without $f_L(u)$ being ignited. By Proposition \ref{P:SI-to-BR}, this implies that at time $t$, \eqref{E:P-contained} must hold because of the spread of $\fB_1$ along an edge $(u, v)$ with $u \in \fB_1(s)$ for some $s < t$, and $Pu \ne Pv$.

Now, if $t < \tau$, then $(u, v)$ is necessarily in the region where 
\begin{equation}
\label{E:xiL-kappa}
\xi_L \ka_1(u, v)/\ka_0 \ge 2 \xi_{2L} \ka_2(Pu, Pv)/\ka_0.
\end{equation}
Note that this requires the bound $R_2 + 2L < r_1$ which implies that the finer colouring process $(\fR_1, \fB_1)$ does not start using its medium-speed clocks in the annulus $A(0, r_1, R_1)$ until the region where the coarser process $(\fR_2, \fB_2)$ has reached its fastest speed.

Moreover, we have the colouring time inequality $T_2(Pu) \le T_1(u)$, since otherwise we would have a smaller value of $t$ in \eqref{E:P-contained}. Combining these facts implies that $Pv$ is coloured (either red or blue) prior to when $v$ was coloured, which is a contradiction. Hence $t \ge \tau$.

Next, let $s$ be the first time when 
$$
P^{-1} \fR_2(s) \not \sset \fB_{1, *} \cup \fR_1(s + \xi_{2L}/\ka_0).
$$
To prove the lemma, we just need to check that $s + \xi_{2L}/\ka_0 \ge \tau$. 

Suppose not. For some $w \in \fR_2(s)$ and $v \notin \fB_{1, *} \cup \fR_1(s + \xi_{2L}/\ka_0)$, we have $P v  = w$. By Proposition \ref{P:SI-to-BR}, the box $f_{2L}(w)$ is ignited and is not a blue seed. By Lemma \ref{L:lower-scale-ignition}, this guarantees that all four blocks in $f_L P^{-1} w$ will have an infected particle in them prior to time $T_2(w) + \xi_{2L}/\ka_0$, ensuring that $v$ is coloured in $(\fR_1, \fB_1)$ prior to time $T_2(w) +\xi_{2L}/\ka_0$. Hence $v \in \fB_1(s + \xi_{2L}/\ka_0)$ since $v \notin \fR_1(s + \xi_{2L}/\ka_0)$.

Since $v$ is not a blue seed, it must have been coloured blue by the firing of a clock on an edge $(u, v)$ where $u$ is blue in $(\fR_1, \fB_1)$ and
\begin{equation}
\label{E:T1u}
T_1(u) + \xi_L \ka_1(u, v)/\ka_0  = T_1(v) \le T_2(w) +  \xi_{2L}/\ka_0.
\end{equation}
Now, using that $s + \xi_{2L}/\ka_0 \le \tau$, again the inequality \eqref{E:xiL-kappa} holds for the edge $(u, v)$. By this inequality, \eqref{E:T1u}, and the fact that $T_2(w) \le T_2(Pu) + \xi_{2L} \ka_2(Pu, w)/\ka_0$, we get that $T_1(u) < T_2(Pu)$. This implies \eqref{E:P-contained} with $t = s$, which contradicts that $s < \tau$.
\end{proof}

\begin{theorem}
\label{T:linear-speed-RB-process}
For all small enough $p > 0$, for any $\ep > 0$ we can find parameters $r = r(\ep), R = R(\ep), \chi = \chi(\ep)$ such that the process
$$
(\fR_1, \fB_1) :=(\fR, \fB)[L_p, \chi, r, R]
$$
satisfies all the conditions of Theorem \ref{T:random-red-blue-var-ka} with inner radius $3r \chi/(\xi_{L_p} L_p)$, outer radius $R/L_p$, and event $\cE$ satisfying $\P [\cE] \ge 1 - \ep$. Here $L_p$ is defined as in Corollary \ref{C:stoch-dom}.
\end{theorem}

\begin{proof}
First, let $c$ be an absolute constant larger than every constant from Section \ref{S:SSP}. Let $p$ be chosen small enough so that $p \le p_0$ in Theorem \ref{T:random-red-blue-var-ka} and so that $c p < 1$. Let $\ep > 0$, and let $\de  > 0$ be small enough so that $1 - c \de > 1 - \ep/4$. Let $\{L^1, \dots, L^k \} = \{2^k : k \in \N\} \cap [L_p, L_\de]$. Then the SSP
$$
(\fR_k, \fB_k) := (\fR, \fB)[L^k, 8\xi_{L^k}, 0, 0]
$$
has a blue seed set $\fB_*$ that is stochastically dominated by an i.i.d.\ Bernoulli process with parameter $\de$ by Corollary \ref{C:stoch-dom}, and so by Theorem \ref{T:random-red-blue}, the event
\begin{equation}
\label{E:A0-construct}
\cE_0 := \{[\fC_k]^c \sset \fR_1(\infty)^c, \; [\fC_k] \text{ has only bounded components} \}
\end{equation}
has probability probability at least $1 - \ep/2$. Here and throughout the proof we use subscripts on objects, e.g. $\fC_i, \fD_i$, to denote that they are associated to a particular SSP $(\fR_i, \fB_i)$. We will now recursively construct processes 
$$
(\fR_i, \fB_i) := (\fR, \fB)[L^i, \chi_i, r_i, R_i]
$$
where the parameters are chosen according to the following rules. Set $r_k = R_k = 0, \chi_k = 8\xi_{L^k}$, and for $i < k$, we require that
\begin{itemize}
	\item $\chi_i = 8 \chi_{i+1}$
	\item $r_i$ is chosen large enough given $R_{i+1}$ so that $r_i \ge R_{i+1} + L^{i+1}c,$ and
	\begin{equation}
	\label{E:exp-lf-rg}
	\exp \lf( \log (cp) \exp\lf(\frac{\log (\chi_ir_i/(\xi_{L^i}L^{i+1}))}{c \log \log (\chi_ir_i/(\xi_{L^i}L^{i+1}))}\rg)\rg) \le \frac{\ep}{5k}.
	\end{equation}
	\item For each $i$, let
	\begin{align*}
\tau_i &= \min \{t \in [0, \infty) : f_{L^i}(\fR_i(t) \cup \fB_i(t)) \not \sset D(0, R_i) \} \\
\sig_i &= \min \{t \in [0, \infty) : D(0, 3\chi_ir_i/(\xi_{L^{i}} L^{i+1})) \sset \fR_{i+1}(t) \cup \fB_{i+1}(t) \}.
	\end{align*}
	Then $R_i$ is chosen large enough given $r_i$ so that $R_i \ge 3 \chi_i r_i/\xi_{L^i}$ and 
	\begin{equation}
	\label{E:tau-sig}
	\P( \tau_i \ge \sig_i) \ge 1 - \ep/(5k).
	\end{equation}
\end{itemize}
Since $cp < 1$, we can clearly choose $\chi_i, r_i$ given $\chi_{i+1}, r_{i+1}, R_{i+1}$ that guarantee the first two bullet points above. To see that we can guarantee the third bullet point, first note that $\sig_i$ is deterministically bounded above for a fixed $\chi_i$ by the lower bound on the growth of $\fR_i(t) \cup \fB_i(t)$ from the definition of the clocks $X_{\fR_i}, X_{\fB_i}$. Moreover, the growth bound on the SI model in Theorem \ref{thm:KS.upper.bound} guarantees that $\E \tau_i \to \infty$ with $R_i$. This implies \eqref{E:tau-sig} for large enough $R_i$.  

Next, for each $i$ define
$$
\fB_{\clubsuit, i} = \fB_{*, i} \cap f_{L^i}^{-1} (D(0, r_i \chi_i/\xi_{L^i})^c) \sset \Z^2.
$$
By Corollary \ref{C:stoch-dom}, all of the sets $\fB_{\clubsuit, i}$ are stochastically dominated by i.i.d.\ Bernoulli processes with parameter $p$. Let $\fD_i$ be defined from $\fB_{\clubsuit, i}$ as in the statement of Theorem \ref{T:random-red-blue-var-ka}, and let $\fF_i = P \fD_i \cup \fD_{i+1}$ be as in Lemma \ref{L:random-two-seed-sets}. Then by the bound on $P_{a,b}$ in Lemma \ref{L:random-two-seed-sets} and the bound in \eqref{E:exp-lf-rg}, with probability at least $1 - \ep/(5k)$ we can find a simply connected set $W_i \sset \Z^2$ separating $\infty$ from $0$ such that
$$
D(0, 2\chi_i r_i/(\xi_{L^i} L^{i+1})) \sset W_i^c \sset D(0, 3\chi_i r_i/(\xi_{L^i}L^{i+1}))
$$
and such that $\del W_i \cap [\fF_i] = \emptyset$. We let $V_i = P^{-1} W_i$, so that 
\begin{equation}
\label{E:DV-containment}
D(0, 2\chi_i r_i/(\xi_{L^i} L^{i})) \sset V_i^c \sset D(0, 3\chi_i r_i/(\xi_{L^i}L^{i}))
\end{equation}
Let $\cB_i$ be the event where such a $V_i$ exists, and be $\cC_i$ be the event in \eqref{E:tau-sig}. For the remainder of the proof, we work on the intersection $\cE$ of all the events $\cE_0, \cB_i, \cC_i, i = 1, \dots, k - 1$, which has probabability at least $1 - \ep$ by a union bound. We will inductively show that on this event, all of the processes
$(\fR_i, \fB_i), i =1, \dots, k-1$ satisfy the conditions of Theorem \ref{T:random-red-blue-var-ka} with $\cE, V_i$ as above, $r = 3r_i \chi_i/(\xi_{L^i} L^{i})$, and $R = R_i/L^i$. The conclusion of Theorem \ref{T:linear-speed-RB-process} is the case when $i = 1$.

We start with the inductive step. Let $i < k - 1$, and suppose that  Theorem \ref{T:random-red-blue-var-ka} holds for $(\fR_{i+1}, \fB_{i+1})$. We want to check the conditions of that theorem for $(\fR_i, \fB_i)$. The first two bullet points of that theorem are clear, so we just need to check the third bullet point. Our goal will be to appeal to Lemma \ref{L:interacting-RB}.

First, the first two bullet points above guarantee the pair $(\fR_i, \fB_i), (\fR_{i+1}, \fB_{i+1})$ satisfies the conditions of Lemma \ref{L:adjacent-boxes}. By that lemma, we have that
 $$
 P^{-1} \fR_{i+1}(t) \sset \fB_{i, *} \cup \fR_1(t + \xi_{L^{i+1}}/\ka_0)
 $$
 for all $t \le \tau_i$. Since $\sig_i \le \tau_i$ on the event $\cE$, this implies that
 \begin{equation}
 \label{E:fR-fRi}
 (P^{-1} \fR_{i+1}(\infty)) \cap D(0, 3\chi_ir_i/(\xi_{L^{i}} L^{i}))  \sset \fB_{i, *} \cup \fR_1(\infty).
 \end{equation}
 This implies that the pair $(\fR_i, \fB_i), (\fR_{i+1}, \fB_{i+1})$ satisfies the interaction assumption of Lemma \ref{L:interacting-RB} with $r = 3\chi_ir_i/(\xi_{L^{i}} L^i)$. Moreover, the set $V_i$ that we have defined satisfies
 $$
 	V_i^c \sset D(0, 3\chi_ir_i/(\xi_{L^{i}} L^i)),\qquad \del V_i \cap [\fD_i] = \emptyset 
 	$$
 	so to appeal to Lemma \ref{L:interacting-RB}, we just need to check that $ P \del V_i \sset \fR_{i+1}(\infty)$. First, $P \del V_i \sset \del W_i$, so it suffices to prove the same claim for $\del W_i$. Next, on the event $\cE$, by the inductive hypothesis we have that 
 	\begin{equation}
 	\label{E:Vi+1}
V_{i+1} \cap [\fC_{i+1}]^c  \sset \fR_{i+1}(\infty),
 	\end{equation}
 	so it suffices to check that 
$
\del W_i \sset V_{i+1} \cap [\fC_{i+1}]^c.
$
 By construction, $\del W_i$ avoids $[\fC_{i+1}]$, so we just need to show that $\del W_i \sset V_{i+1}$. Indeed, this follows from the bound \eqref{E:DV-containment} at levels $i$ and $i + 1$, and the fact that $3\chi_i r_{i+1}/\xi_{L^{i+1}} \le R_{i+1} \le r_i$. This completes the proof of the inductive step.
 
 The base case for $i = k - 1$ is essentially the same. The only difference is that instead of appealing to the inductive hypothesis to get \eqref{E:Vi+1}, instead we can simply appeal to the fact that we are working on $\cE_0$, see \eqref{E:A0-construct}. 
\end{proof}

We summarize the important conclusions of Theorem \ref{T:linear-speed-RB-process} and Theorem \ref{T:random-red-blue-var-ka} in the next corollary for the convenience of the reader.

\begin{corollary}
	\label{C:for-next-sect}
	For all small enough $p > 0$, for every $\ep > 0$ there exist finite parameters $\chi, r, R$ such that the SSP
	$
	(\fR, \fB) := (\fR, \fB)[L_p, \chi, r, R]
	$
	satisfies the following conclusions:
	\begin{enumerate}[label=(\roman*)]
		\item The subset of blue seeds 
		$$
		\fB_\clubsuit = \fB_* \cap D(0, r \chi/(\xi_{L_p} L_p))
		$$
		 is stochastically dominated by an i.i.d.\ Bernoulli process with parameter $p$.
		 \item Letting the set $\fD$ be defined from $\fB_\clubsuit$, there exists an event $\cE_\ep$ with $\P[ \cE_\ep] \ge 1 - \ep$ such that on $\cE_\ep$,
		 $$
		 [\fD]^c \cap D(0, R/L_p)^c \sset \fR(\infty)
		 $$
		 and all components of $\fB(\infty)$ are bounded.
		 \item There exists a constant $\de > 0$ that depends on $p$ but not on $\ep, \chi, r,$ or $R$ such that on $\cE_\ep$, $D(0, \de t) \sset [\fR(t)]$ for all large enough $t$.
	\end{enumerate}
\end{corollary}


\subsection{Infection spread after colouring}

In this subsection we establish that once a block has been coloured, all its particles are infected after a short time with high probability.  Let us consider a fixed $p > 0$ and side length $L_p$ satisfying the conditions of Corollary \ref{C:for-next-sect}. Fix $\ep > 0$, and let all parameters and sets (e.g. $\cE_\ep, R, \fD$) be as in Corollary \ref{C:for-next-sect}.
Let $\sD=\sD_{B,r}$ denote the set of blocks in $\sB_L$ contained in $D(B,r)$. Define the event
\[
\cJ_{B,r} =\bigg \{|\{B'\in \sD_{B,r}\cap f_L[\fD]\}| \leq \frac13 |\sD_{B,r}| \bigg\}.
\]
We can bound $\P[ \cJ_{B,r}]$ by \eqref{E:PDxr} in Lemma \ref{L:largest-scale}.
%
We use this to show that there are enough infected particles from red blocks to infect all particles quickly. 
\begin{proposition}\label{p:particleInfectTime}
Let $\cK_B$ be the first time that all particles coloured in $B$ are infected.  For any $C_1>0$ there exists $C_2>0$ depending on $p, \ep$ such that for any block $B$ we have
\[
\P[\cK_B - \tau_B \geq y, \cE_\ep] \leq C_2 y^{-C_1} \qquad \text{ for all } y > 0.
\]
\end{proposition}

\begin{proof}
It is enough to consider $y$ very large. Letting
\[
\cO:=\bigcap_{B'\in \sD_{B,2y}}\bigcap_{a\in H_{B'}}\{\sup_{0 \le t\leq  y^{2/3}} |a(t)-a(0)| \leq y^{7/20}\}
\]
we have that
\[
\P[\cO] \geq \exp(-y^{1/100})
\]
by standard random walk estimates and a union bound over the number of particles in $H_{B'}$ by Lemma~\ref{l:HBsize}.  
Letting
\[
\cQ:=\bigcap_{a\in H_{B}}\{\sup_{0 \le t\leq  y} |a(t)-a(0)| \leq y^{11/20}\}
\]
we similarly have that
\begin{equation}
\label{E:Q}
\P[\cQ] \geq \exp(-y^{1/100}).
\end{equation}
We let
\[
\cZ_{B,y}=\Big\{\forall B'\in \sD_{B,y^{5/9}}, |\{a\in{\bf I}_{\tau_B+y^{3/5}}:\oa(\tau_B+y^{3/5})\in D(B',2 y^{2/5}) \}|\geq \frac12|\sD_{B', y^{2/5}}| \Big\},
\]
and will work to bound $\P[\cZ_{B,y}^c, \cE_\ep]$. Here recall that $\bf I_t$ denotes the set of infected particles at time $t$.

Let $B' \in \sD_{B,y^{5/9}}$. On $\cJ_{B',y^{2/5}}$, a $\frac23$-fraction of the blocks in $\sD_{B,y^{5/9}}$ are outside $f_L[\fD]$.  On the event $\cE_\ep$, all $B'' \in \sD_{B',y^{2/5}} \cap f_L[\fD]^c$ that do not intersect $D(0, R)$ are red  by Corollary~\ref{C:for-next-sect} and generate at least  one infected particle. Hence, for large enough $y$, on $\cJ_{B',y^{2/5}} \cap \cE_\ep$ at least $\frac12$ of the blocks in $\sD_{B', y^{2/5}}$ generate an infected particle. 

On the event $\cO$, those particles travel at distance at most  $y^{7/20}$ by time $\tau_{B'}+y^{2/3}$ and so remain within $D(B',2 y^{2/5})$.  Since the colouring process spreads linearly, we have $|\tau_B-\tau_{B'}|\leq c' y^{5/9}$ for a constant $c'$ depending on the parameters $L_p, \chi, r, R$.
Therefore $\tau_{B'} \leq \tau_B + y^{3/5} \leq \tau_{B'} + y^{2/3}$ for large $y$, and so
\begin{align*}
&\cJ_{B',y^{2/5}} \cap \cO \cap  \cE_\ep \\
&\qquad \subset \Big\{ |\{a\in{\bf I}_{\tau_B+y^{3/5}}:\oa(\tau_B+y^{3/5})\in D(B',2 y^{2/5}) \}|\geq \frac12|\sD_{B', y^{2/5}}| \Big\}.
\end{align*}
Hence bounding $\P [\cJ_{B,r}]$ by \eqref{E:PDxr} in Lemma \ref{L:largest-scale}, we get
\begin{equation}
\label{E:P-bound}
\P[\cZ_{B,y}^c, \cE_\ep] \leq 1 - \P[\bigcap_{B'\in \sD_{B,y^{5/9}}} \cJ_{B,y^{2/5}} \cap \cO ] \leq C_2 y^{-C_1}.
\end{equation}
On the event $\cZ_{B,y}$ we have that for all $x\in D(B,y^{11/20})$ and $s\in[\frac12 y,y-y^{3/5}]$,
\[
\sum_{a\in {\bf I}_{\tau_B+y^{3/5}}} P_{\oa(\tau_B+y^{3/5}),x}^s \geq c L^{-2},
\]
where $P^s_{x, y}$ is the transition probability for a random walk to go from $x$ to $y$ in time $s$. Therefore by Lemma~\ref{l:hittingNumberSimple}, if $a\in H_B$ then
\[
\P[\cZ_{B,y},\cQ,a\not\in {\bf I}_{\tau_B+y}] \leq C_2 y^{-C_1}.
\]
Taking a union bound over $a\in H_B$ using the estimate on $|H_B|$ in Lemma \ref{l:HBsize} and combining this with the estimates in \eqref{E:Q} and \eqref{E:P-bound} completes the proof. 
\end{proof}
\subsection{Proof of Theorem~\ref{thm:main} for $d = 2$}

Let $L_p \in \N, \de > 0$ be as in Corollary \ref{C:for-next-sect}, and let $\gamma=\tfrac{1}{10} \de L_p$. Let $\cY_n$ be the event that a particle that starts in $D(0,2\gamma n)^c$ enters $D(0,\gamma n)$ before time $n+1$.  By standard random walk estimates, $\P[\cY_n] \le e^{-cn}$ for large enough $n$.  By Corollary \ref{C:for-next-sect}, for all large enough $n \in \N$, all boxes in the set $D(0, 2\gamma n)$ will be coloured by time $\frac12 n$.  Define
\[
\cX_n = \bigcup_{B \in \sB_{L_p}:B\cap D(0,2 \gamma n)\neq \emptyset}\Big\{ \cK_B-\tau_B >\frac{n}{2} -1 \Big\}.
\]
Letting $\ep > 0$, by Proposition~\ref{p:particleInfectTime} we have that
\[
\P[\cX_n, \cE_\ep] \leq C n^{-100}
\]
for a constant $C$ that depends on $\ep$ but not on $n$.
On the event $\cX_n^c$ all particles that start in $D(0,2\gamma n)^c$ are infected before time $n-1$.  If $(\cX_n\cup \cY_n)^c$  holds then all particles in $D(0,\gamma t)$ at time $t$ are infected for times  $t\in[n-1,n]$. By the above bounds and the Borel-Cantelli Lemma, $(\cX_n\cup \cY_n)^c$ holds only finitely often a.s. on $\cE_\ep$. Since $\ep > 0$ was arbitrary and $\P[ \cE_\ep] \ge 1 - \ep$ by Corollary \ref{C:for-next-sect}, this completes the proof.

\section{Linear growth in dimension $d = 1$}
\label{S:dim-1}
In dimension $d=1$, we cannot appeal to the framework of Sidoravicius-Stauffer percolation to prove that the infection survives forever. However, unsurprisingly a simpler proof idea works in this case. The key is again to use a Poisson decomposition of the process with appropriate notions of red and blue (fast and slow) blocks. However, since the process cannot travel around slow blocks, if we encounter a slow block we will instead look for a fast block at a larger scale.

In this section, all constants will depend on the diffusion rates $D_S, D_I$ and the density $\mu$.
Fix a small constant $\de > 0$ and a large constant $L \in \N$. How large or small we need to take these values will be made clear in the proof. As it turns out, we will need to choose a fixed $\de > 0$ so that two of the steps in the forthcoming Lemma \ref{L:Ki-bd} go through, and then $L$ will be chosen so that $L \de^2$ is larger than some big $D_S, D_I, \mu$-dependent constant.

 For every $a \in \Z$, let $\tau_a$ be the first time when an infected particle reaches site $a$, and let $\tau_0 = 0$. For $a, k \in \N$, let $\fR(a, k)$ be the event where the following two conditions are satisfied.
\begin{enumerate}[label=(\roman*)]
	\item At every time $t \in [\tau_a, \tau_a + k^2 L^2]$ there are at least $\de k L$ total particles in the interval $(a, a + kL]$. Note that at time $\tau_a$, all of these particles are necessarily susceptible.
	\item $\tau_{a + kL} \le \tau_a + k^2 L^2$.
\end{enumerate}
We think of $\fR(a, k)$ as a `red' event where the process moves sufficiently quickly at scale $k$ through the box $(a, a + kL]$. Letting $S_0 = 0, K_0 = 0$, we now recursively define scales and increments $K_i, S_i$ by
$$
K_i = \inf \{k \in \N, k \ge K_{i-1} - 1 : \fR(S_{i-1}, k) \text{ holds} \}, \quad S_i = S_{i-1} + L K_i = \sum_{j=1}^i L K_j
$$
Also, let $\cF_i$ be the $\sig$-algebra generated by all trajectories of all infected particles up to time $\tau_{S_i}$. Our goal is to show that the sizes of the $K_i$ are well-controlled. The key to doing this is the following observation:
\begin{itemize}
	\item Let $Y(t)$ be the location of the rightmost infected particle in the process. Conditionally on $\cF_i$, the distribution of susceptible particles in the interval $(S_i, \infty)$ at time $\tau_{S_i}$ is a Poisson process of varying intensity $\mu P^b_i, b \in (S_i, \infty)$, where
	$$
	P^b_i = \P( X(t) + b - Y(\tau_{S_i} - t) > 0 \quad \text{ for all } t \in [0, \tau_{S_i}] \; | \; \cF_i),
	$$ 
	where $X$ is an independent continuous time random walk of rate $D_S$ started at $0$.
\end{itemize}
This observation is the analogue of the Poisson description of the SI process used in prior sections. The construction of the $K_i, S_i$ above allows us to get a lower bound on the probabilities $P^b_i$.

\begin{lemma}
	\label{L:Pb-lower-bd}
There is a constant $c> 0$ depending only on $D_S$ such that for $k \ge 2$ we have
$$
P^b_k \ge c \mathbf{1}(b - S_k \ge LK_k/8).
$$
\end{lemma}

The key point in the proof of Lemma \ref{L:Pb-lower-bd} is that
\begin{equation}
\label{E:ij-diff}
K_j \le K_i + i - j \qquad \text{ for all } j \le i \in \N.
\end{equation}
To take advantage of this, we prove a simple lemma about sequences.

\begin{lemma}
	\label{L:sum-si}
	Let $s_i$ be a sequence of natural numbers such that $s_j \le s_i + i - j$ for all $j \le i \in \N$. Then
	\begin{equation}
	\label{E:sum-si}
	\sum_{i=1}^k s_i^2 \ge t \qquad \implies \qquad \sum_{i=1}^k s_i \ge \frac{t}{2s_k} \wedge \frac{t^{2/3}}2.
	\end{equation}
\end{lemma}

\begin{proof}
	Let $s_j = M$ be the largest element of $\{s_1, \dots, s_k\}$. We have
	\begin{equation}
	\label{E:si-bound}
	\bigslant{\sum_{i=1}^{k} s_i}{\sum_{i=1}^{k} s_i^2} \ge \frac{1}M,
	\end{equation}
	so if $M \le 2 s_k$, \eqref{E:si-bound} implies the lemma. Now suppose $M > 2 s_k$. Without loss of generality, we may assume that 
	$$
	\sum_{i = j+1}^{k} s_i^2 \le t,
	$$
	since otherwise we could remove $s_1, \dots, s_j$ and decrease the right side sum in \eqref{E:sum-si}. Now, the conditions of the lemma guarantee that the interval $\{s_k, s_k + 1 \dots, M-1\}$ is contained in the set $\{s_{j+1}, \dots, s_k\}$, so
	\begin{equation}
	\label{E:M-bound}
	\frac{1}{8} M^3 \le \frac{1}{2} (M - s_k)((M-1)^2 + s_k^2) \le \sum_{i = j+1}^{k} s_i^2 \le t.
	\end{equation}
	For the first inequality in \eqref{E:M-bound}, we have used that $M > 2 s_k$, and hence also $M \ge 3$. Combining \eqref{E:M-bound} and \eqref{E:si-bound} then yields the result.
\end{proof}

\begin{proof}[Proof of Lemma \ref{L:Pb-lower-bd}]
Let $\tilde Y(t) = Y(\tau_{S_k}) - Y(\tau_{S_k} - t)$. The main step in proving the lemma is to show that for all $t$, 
\begin{equation}
\label{E:tilde-Y}
\tilde Y(t) \ge \frac{t^{2/3}}{6 K_k^{1/3} L^{1/3} } \mathbf{1}(K^2_k L^2 < t \le \tau_{S_k}).
\end{equation}
Given \eqref{E:tilde-Y}, letting $X(t)$  be an independent rate-$D_S$ continuous time random walk started at $0$, for a starting location $b$ with $b - S_k \ge L K_k/8$ we have that
\begin{align*}
   P^b_k &\ge \P\Big(\forall t > 0 : X(t) < \frac{L K_k}{8} + \frac{t^{2/3}}{6 K_k^{1/3} L^{1/3} } \mathbf{1}(K^2_k L^2 < t) \Big) \\
   &= \P\Big(\forall s > 0 : \frac{X(L^2 K_k^2 s)}{L K_k} < \frac{1}{8} + \frac{s^{2/3}}6 \mathbf{1}(1 < s)\Big).
\end{align*}
By standard random walk estimates, this probability is bounded below by a constant $c > 0$; this constant is independent of the choice of $K_k$
since the scaling factor $L K_k$ is bounded away from $0$. 

To complete the proof of the lemma, it just remains to show \eqref{E:tilde-Y}.
Let $K^2_k L^2 < t \le \tau_{S_k}$. Observe that
\begin{equation}
\label{E:YtSk}
\tilde Y(t) \ge S_k - S_{j(t)} = \sum_{i = j(t) + 1}^{k} L K_i, \quad \text{ where } \quad j(t) = \min \{j \in \N: \tau_{S_j} > \tau_{S_k} - t \}.
\end{equation}
Next, observe that
\begin{equation}
\label{E:tauSk}
t \le \tau_{S_k} - \tau_{S_{j(t) - 1}} = \sum_{i= j(t)}^k \tau_{S_i} - \tau_{S_{i-1}} \le \sum_{i= j(t)}^k K^2_i L^2.
\end{equation}
Since $t > K^2_k L^2$, \eqref{E:tauSk} implies that $j(t) < k$. Combined with \eqref{E:ij-diff}, this gives
\begin{equation}
\label{E:single-sum}
\sum_{i = j(t) + 1}^{k} L K_i \ge \frac{1}{3} \sum_{i = j(t)}^{k} L K_i.
\end{equation}
By looking at \eqref{E:YtSk}, \eqref{E:tauSk}, \eqref{E:single-sum}, to lower bound $\tilde Y(t)$, we see that we want to bound the right side of \eqref{E:single-sum} below subject to the constraint that $\sum_{i=j(t)}^k K_i^2 \ge t/L^2$. By \eqref{E:ij-diff}, this puts us in the setting of Lemma \ref{L:sum-si}. Therefore
$$
\tilde Y(t) \ge \frac{t}{6 L K_k} \wedge \frac{t^{2/3}}{6 L^{1/3}} \ge \frac{t^{2/3}}{6 L^{1/3} K_k^{1/3}},
$$
where in the final equality, we use that $t > K_k^2 L^2$. This gives \eqref{E:tilde-Y}.
\end{proof}

We use Lemma \ref{L:Pb-lower-bd} to bound the $K_i$.
\begin{lemma}
\label{L:Ki-bd}
There exists a constant $c>0$  such that for $i \ge 2$ and $n, m \in \N$ with $n \ge m-1$ we have
$$
\P(K_i > n, K_{i-1} = m \mid  \cF_{i-1}) \le c n^{-4} \mathbf{1}(K_{i-1} = m).
$$
\end{lemma}

In the proof, constants $c, c', c''$ depend only on $D_S, D_I, \mu$.

\begin{proof}
We work conditionally on $\cF_{i-1}$ and on the event where $K_{i-1} = m$. The conditional probability in the lemma can then be bounded above by the conditional probability of $\fR(S_{i-1}, n)^c$. 

First, as long as $\de$ was chosen sufficiently small, since $n \ge m -1$, under this conditioning the number of particles that stay in the spatial interval $[S_{i-1} + n L/3, S_{i-1} + 2nL/3]$ throughout the time interval $[\tau_{S_{i-1}}, \tau_{S_{i-1}} + n^2 L^2]$ stochastically dominates a Poisson random variable of mean at least $2 \de n L$. This uses the density bound in Lemma \ref{L:Pb-lower-bd}.
Therefore condition (i) fails with probability at most  $\exp (- c \de n L) \le e^{-n}$.

Now set $x = \sqrt{\de} nL \log^{1/2} (n + 1)$. Condition (ii) is implied by the following three events:
\begin{enumerate}[label=\Roman*.]
	\item At time $\tau_{S_{i-1}} + n^2L^2/2$, at least one infected particle is in the region $[S_{i-1} - x, \infty)$.
	\item At time $\tau_{S_{i-1}} + n^2L^2/2$, at least $\de n^{1/2} L$ particles that were uninfected at time $\tau_{S_{i-1}}$ lie in the interval $J=(S_{i-1} - 2x, S_{i-1} - x]$.
	\item Of these particles, at least one is in the region $[S_{i-1} + nL, \infty)$ at time $\tau_{S_{i-1}} + n^2 L^2$.
\end{enumerate}
We estimate the probability of events I, II, and III failing. For I, by condition (i) for the event $\fR(S_{i-1}, m)$, there are at least $\de m L$ infected particles in the region $(S_{i-1} - m L, S_{i-1}]$ at time $\tau_{S_{i-1}}$. For the first event to fail, all of these  need to move left by distance at least $x - mL$ in the interval $[\tau_{S_{i-1}}, \tau_{S_{i-1}} + n^2L^2/2]$. For large enough $n$, we have $x - mL \ge x/2 = \sqrt{\delta} n L \log^{1/2}(n+1)/2$. For such $n$ the probability that one individual particle moves left by at least $x - mL$ is $n^{-c' \de}$ for some absolute constant $c' > 0$. Therefore the probability that all particles move left by this much is at most $n^{-c' m \de^2 L}$. As long as $\de^2 L$ is sufficiently large, this probability is at most $O(n^{-4})$.

Next, we show the probability that event II fails also has an upper bound of order lower than $O(n^{-4})$. Since the event (i) fails with probability at most $e^{-n}$, we may assume this event holds.
Given the event in (i), there are at least $\de n L$ susceptible particles in the region $(S_{i-1}, S_{i-1} + nL]$ at time $\tau_{S_i}$. Each of these has probability at least
$$
c''\exp (- c'x^2 / (n^2 L^2)) \ge c'' n^{-c' \de}
$$
of being in the interval $J$ at time $\tau_{S_{i-1}} + n^2 L^2/2$. Therefore the number of particles that were uninfected at time $\tau_{S_{i-1}}$ that are in $J$ at this time stochastically dominates a binomial random variable with $\de n L$ trials and success probability $n^{-c' \de}$. As long as $\de$ was chosen small enough so that $c' \de \le 1/4$ and $\de L$ is sufficiently large, a standard estimate on the lower tail of a binomial random variable then shows that event II fails with probability less than $O(n^{-4})$. A similar argument shows that the event III also fails with this probability. Putting all this together shows that condition (ii) fails with total probability $O(n^{-4})$.
\end{proof}

\begin{lemma}
	\label{L:step-1}
	Almost surely, $K_1 < \infty$.
\end{lemma}

The proof of Lemma \ref{L:step-1} is essentially the same as the proof of Lemma \ref{L:Ki-bd}, so we omit it. The only difference is that for bounding the probability of event I in the proof above, we can only assume the existence of \textbf{one} infected particle (rather than $\de m L$). This reduces the upper bound on $\P(K_1 > n)$ from $O(n^{-4})$ to $O(n^{-\ep})$ for some $D_S, D_I, \mu$-dependent positive number $\ep$.

\begin{prop}
	\label{P:linear-growth}
	There exists $c> 0$ depending on $D_S, D_I, \mu$ such that almost surely,
	$$
	\limsup_{|a| \to \infty} \frac{\tau_a}{a} \le c < \infty. 
	$$ 
\end{prop}

\begin{proof}
	First, by symmetry it suffices to prove the bound as $a \to \infty$. Next, for $k \in \N, k \ge 2$, we can write
	\begin{equation}
	\label{E:Mi-bound}
	\frac{\tau_{S_k}}{S_{k-1}} \le \frac{\sum_{i=1}^k K_i^2 L^2}{ \sum_{i=1}^{k-1} K_i L} \le \frac{L}{k-1} \sum_{i=1}^k K_i^2.
	\end{equation}
	Here the first inequality uses condition (ii) in the definition of $\fR(a, k)$, and the second equality uses the trivial bound that $\sum_{i=1}^{k-1} K_i L \ge (k-1) L$. 
	By Lemma \ref{L:Ki-bd}, 
	the sequence $(K_i, i \in \N)$ is stochastically dominated by a recursively defined sequence $L_i, i \in \N$, where $L_1 = M_1$ and $L_{i+1} = (L_i - 1) \vee M_{i+1}$ for a sequence $M_i, i \in \N$, where the $M_i, i \ge 2$ are i.i.d.\ with finite third moment. Now, in a coupling where $K_i \le L_i$ for all $i$, we have
	$$
	\sum_{i=1}^k K_i^2 \le  \sum_{i=1}^k L_i^2 = \sum_{i=1}^k \max_{1 \le j \le i} [(M_j - (i-j))^+]^2 \le \sum_{i=1}^k M_i^3.
	$$
	This bound and the law of large numbers shows that the limsup of the right side of \eqref{E:Mi-bound} is almost surely bounded above by some constant $c$, so 
		\begin{equation}
		\label{E:ktoinfty}
	\limsup_{k \to \infty} \frac{\tau_{S_k}}{S_{k-1}} \le c < \infty. 
	\end{equation}
	Finally, for every $a \in \N$ we have $\tau_a/a \le \tau_{S_{k(a)}}/S_{k(a)-1}$, where $k(a)$ is such that $S_{k(a) - 1} \le a \le S_{k(a)}$. Finally, $k(a) \to \infty$ with $a$, so \eqref{E:ktoinfty} implies the result.
%
%
%
%
\end{proof}

Theorem \ref{thm:main} in dimension $1$ is a rephrasing of Proposition \ref{P:linear-growth}.

\section{Proof of coupling}
\label{s:coupling}

In this section, we prove that the construction of the SI process given in Section~\ref{S:SI-colouring} is valid. All notation is as in that section, and for the proof in this section, $L$ is a fixed side length and all blocks belong to $\sB = \sB_L$. The parameters $r, R, \chi$ are also fixed.

The natural construction of the SI process is built from the Poisson process $\sP$ on $\fW$ along with a single ignition trajectory $W_{B_0,\operatorname{ig}}$ for the initially infected particle, where $B_0$ is the block containing the origin. This is the construction of the SI process given in \cite{kesten2005spread}. In the arXiv version \cite{kesten2003spread} of that article they show that the construction gives a strong Markov process. By the strong Markov property, the following process is equal to the original SI process in distribution: run the SI process according to $\sP, W_{B_0,\operatorname{ig}}$ but at every time a particle becomes the ignition particle for a block $B$ at a site $x$, replace its trajectory for $s \in [\tau_B, \infty)$ with $x + W_{B,\operatorname{ig}}(\beta (s- \tau_B))$. It is this latter process that we will couple with the more complex block construction.

For $t\in\R$ we define the time shift function $\psi_t:\fW\to\fW$ by $(\psi_t (w))(s) = w(s-t)$.  For $A\subset \fW$ we let $\psi_t(A):=\{\psi_t(w):w\in A\}$.
We view the Poisson processes $\sP_B, B \in \sB$ taken together as a Poisson process $\sP^*$ on $\sB\times\fW$. For $A\subset \fW$ we let $\sP\mid_A$ denote the restriction of $\sP$ to $A$.  We will show that there is a natural coupling of $\sP$ and $\sP^*$ such that they generate the same SI process $X_t^*$ when using the same set of ignition trajectories.

Let $C_t \subset \Z^2$ be the set of vertices in coloured blocks at time $t$ in the SI process generated by $\sP$ and all ignition trajectories.  Let $\cG_0$ denote the $\sigma$-algebra generated by the ignition particle trajectories of all the blocks. Define the set of paths that have entered the coloured region by time $t$ as
\[
A_t = \{w\in\fW: \exists t'\in [0,t), w(t')\in C_{t'}\}.
\]
Let $\cG_t$ be the filtration generated by $\sP\mid_{A_t}$ and $\cG_0$, which is independent of $\sP$.  
By construction, the map $t \mapsto A_t$ is left continuous; we let $A_t^+= \bigcap_{t'>t} A_{t'}$ be its right continuous version.
We partition $A_t$ into sets 
\[
A_{B,t} = \{w\in A_t: t' = \min\{ s\in [0,t), w(s)\in C_s\},w(t')\in B\}.
\]
This is the set of paths which first enter the coloured region in block $B$.  Our construction ensures that $C_t$ is adapted to $\cG_t$ since  $C_t$ grows either by infected particles from previously coloured blocks entering new blocks or by its spread to neighbouring blocks after time $\xi$ is elapsed. We define the random map
\[
\Psi:\fW\to \sB\times \fW
\]
such that for $w\in A_{B}$,
\[
\Psi(w)=(B,\psi_{\tau_B}(w)).
\]
We will couple $\sP$ and $\sP^*$ such that  $w\in \sP$ if and only if $\Psi(w)\in \sP^*$.  The remaining point process  $\sP^*\setminus\Psi(\fW)$ is not used in the construction and this can be set independently of $\sP$ in the coupling.  To see that this gives the correct coupling note that  a particle with trajectory $w(t)$ is first coloured in box $B$ corresponds to a particle with trajectory $w(t-\tau_B)$ in $\sP_B$.

Now let us see that, given $\cG_0$, two different Poisson processes $\sP \neq \sP'$ cannot give rise to the same $\sP^*$ in the coupling.  Let $C_t'$ be the coloured region given by the process generated by $\sP'$ and define $A_t'$ analogously.  Define the stopping times
\[
T_p =\inf_t \{t: \sP\mid_{A_t} \neq \sP'\mid_{A'_t}\},\qquad T_c= \inf_t \{t: C_t \neq C_t'\}
\]
For $t<T_p  \wedge T_c$ we have that the processes of coloured particles must be equal and so $X_t^* = X_t^{*\prime}$.  Suppose that $T_c=T_p  \wedge T_c$.  The processes $C_t,C_t'$ are right continuous so we must have $C_{T_c}\neq C_{T_c}'$.  Suppose that $B\sset C_{T_c}$ but $B\not\sset C_{T_c}'$.  This can only happen if a particle coloured at some time $T_a<T_c$ in $X_t^*$ enters $B$ at time $T_c$. Now, that particle must also be present in $X_{T_a}^{*\prime}$ since $\sP\mid_{A_{T_a}} = \sP'\mid_{A'_{T_a}}$ and so must make the same jump into $B$ at the same time.  This is a contradiction, so $T_p<T_c$. As a consequence of this, $A_t=A_t'$ for $t\leq T_p$ and $A^+_{T_p} = A^{+ \prime}_{T_p}$.

Since the set $\{t:w\in A_t\}$ is open on the left for any path $w$, we have that  $\sP\mid_{A_{T_p}} = \sP'\mid_{A_{T_p}}=\sP'\mid_{A'_{T_p}}$.  So we must have $\sP\mid_{A^+_{T_p}} \neq  \sP'\mid_{A^+_{T_p}}$.  Suppose $w$ is some trajectory in $\sP\mid_{A^+_{T_p}\setminus A_{T_p}}$ but not in $\sP'\mid_{A^+_{T_p}\setminus A_{T_p}}$.  It is coloured in some block $B\subset C_{T_p}$ and since $B$ was coloured before time $T_c$ we have that $\tau_B=\tau'_B$.  But then $(B,\psi_{\tau_B}(w))\in \sP^*$ so if $\sP'$ can also be coupled with $\sP^*$ then $w$ must be in $\sP'$ as well, which is a contradiction. 

Therefore the map from $\sP \mapsto \sP^*$ is one-to-one, and it follows that from the coupling we can reconstruct $\sP$ from $\sP^*$. To complete the proof of the desired coupling, we just need to check that $\sP^*$ is indeed a Poisson process when defined via $\Psi(\sP)$.

\begin{proposition}\label{p:PoissonCoupling}
Under the coupling given by $\Psi$ we have that $\sP^*$ is a Poisson process on $\sB\times \fW$.
\end{proposition}



\begin{proof}
We let $\sP^*_0$ be a Poisson process on $\sB\times \fW$ which is independent of $\sP$. We will give an inductive construction to show that $\sP^*$ is equal in distribution to $\sP^*_0$.
Let $\cG^+_t$ be the filtration generated by $\sP\mid_{A_t^+}$ and $\cG_0$.
We define $0=T_1< T_2 <\ldots$ to be the times that $C_t$ grows or new particles hit $C_t$.  Formally we say that for $i\geq 1,$ we have
\[
T_{i+1} = \inf\{ t>T_i: C_t\neq C_{T_i}\} \wedge \inf\{ t>T_i: \sP(A_t) > \sP(A_{T_i}^+)\}.
\]
We will define a sequence of point processes $\sP^*_i$ for $i\geq 0$ such that
\begin{equation}\label{eq:sPi.definition}
\sP^*_i\mid_{\Psi(A^+_{T_i})} = \Psi(\sP)\mid_{\Psi(A^+_{T_i})}, \qquad \sP^*_i\mid_{\Psi(A^+_{T_i})^c} =  \sP^*_0\mid_{\Psi(A^+_{T_i})^c}
\end{equation}
where $\Psi(\sP)\mid_{\Psi(A^+_{T_i})}(\Lambda):=\sP(\Psi^{-1}(\Lambda))$ for $\Lambda\subset \Psi(A^+_{T_i})$.  That is, we couple only the set $A^+_{T_i}$ and the remainder of $\sB\times \fW$ is given by $\sP^*_0$. By setting $A_{T_0}^+ := \emptyset$ we can make $\sP^*_0$ is trivially consistent with equation~\eqref{eq:sPi.definition}.  We will prove inductively that  
\begin{equation}\label{eq:sPi.equal.dist}
\sP^*_i \stackrel{d}{=} \sP^*_0.
\end{equation}
Assume \eqref{eq:sPi.equal.dist} holds for some fixed $i$. By the inductive hypothesis and the fact that $\Psi(\sP)\mid_{\Psi(A^+_{T_i})}$ is $\cG_{T_i}^+$-measurable, to show that \eqref{eq:sPi.equal.dist} holds for $i+1$, we just need to show that conditional on $\cG_{T_i}^+$ we have
\begin{equation}
\label{eq:sPi.dist.2}
 \Psi(\sP)\mid_{\Psi(A^+_{T_{i+1}} \smin A^+_{T_{i}})} + \sP^*_0\mid_{\Psi(A^+_{T_{i+1}})^c} \stackrel{d}{=} \sP^*_0\mid_{\Psi(A^+_{T_{i}})^c}.
\end{equation}
Define $C_t^{(i)}=C_t$ for $t\leq T_i$ and for $t>T_i$ we set $C_t^{(i)}$ to be the trajectory taken by $C_t$ if no new particles are found after time $T_i$.  Then $C_t^{(i)}$ is $\cG^+_{T_i}$-measurable for $t\geq T_i$.  Analogously set $A^{(i)}_t = \{w: \exists t'\in [0,t), w(t')\in C_{t'}^{(i)}\}$ which is also $\cG^+_{T_i}$-measurable. We can also define $\tau_B^{(i)}$ and $\Psi^{(i)}$ analogously.  Finally let 
\[
S_{i+1}=\inf\{t > T_i: C_t^{(i)} \neq C_{T_i}\}, \quad U_{i+1} = \inf\{ t>T_i: \sP(A_t^{(i)}) > \sP(A_{T_i}^+)\}.
\]
Then we have that $T_{i+1} = S_{i+1} \wedge U_{i+1}$ since either we find a new particle or $C_t=C_t^{(i)}$ grows. With these definitions, we can see that $C_t^{(i)} = C_t$ for $t \le T_{i+1}$ and hence that $\Psi^{(i)} = \Psi$ on the set $A_{T_{i+1}}^+$. In particular, the left side of \eqref{eq:sPi.dist.2} equals
\begin{equation}
\label{E:Psi-I}
 \Psi^{(i)}(\sP)\mid_{\Psi^{(i)}(A^+_{T_{i+1}} \smin A^+_{T_{i}})} + \sP^*_0\mid_{\Psi^{(i)}(A^+_{T_{i+1}})^c}.
\end{equation}
 Now, $T_{i+1}$ is a stopping time and so $A_{T_{i+1}}^+$ is a stopping set for the filtration $\cG^+_t$. 
 In particular, conditional on $\cG^+_{T_{i+1}}$ we have that $\sP\mid_{(A^+_{T_{i+1}})^c}$ is a Poisson process. Now, conditional on $\cG^+_{T_i}$, $\Psi^{(i)}$ is a Poisson measure preserving map on all of $(A^+_{T_i})^c$, so conditional on $\cG^+_{T_{i+1}}$ we have that
 $$
 \Psi^{(i)}(\sP)\mid_{ \Psi^{(i)}((A^+_{T_{i+1}})^c)}
 $$
 is also a Poisson process on $\Psi^{(i)}((A^+_{T_{i+1}})^c)$. Since $\sP_0$ is independent of all else we then have that conditional on $\cG^+_{T_{i+1}}$,
 $$
 \sP^*_0\mid_{\Psi^{(i)}(\fW)^c} +  \Psi^{(i)}(\sP)\mid_{ \Psi^{(i)}((A^+_{T_{i+1}})^c)} \stackrel{d}{=}  \sP^*_0\mid_{\Psi^{(i)}(A^+_{T_{i+1}})^c}.
 $$
 Since $\Psi^{(i)}(\sP)\mid_{\Psi^{(i)}(A^+_{T_{i+1}} \smin A^+_{T_{i}})}$ is $\cG^+_{T_{i+1}}$-measurable, this implies that conditional on $\cG^+_{T_{i + 1}}$, \eqref{E:Psi-I} is equal in distribution to
 \begin{equation}
 \label{E:Psi-II}
  \sP^*_0\mid_{\Psi^{(i)}(\fW)^c}+\Psi^{(i)}(\sP)\mid_{\Psi^{(i)}((A^+_{T_{i}})^c)}.
 \end{equation}
 Now, since $\cG^+_{T_{i}} \sset \cG^+_{T_{i + 1}}$, the same equality in distribution holds conditional on $\cG^+_{T_{i}}$. Finally, conditional on $\cG_{T_{i}}^+$ we have that $\sP|_{(A^+_{T_i})^c}$ is a Poisson process and $\Psi^{(i)}|_{(A^+_{T_i})^c}$ is a Poisson measure preserving map. Therefore conditional on $\cG_{T_i}^+$, we have that
 $$
\Psi^{(i)}(\sP)\mid_{\Psi^{(i)}((A^+_{T_{i}})^c)} \stackrel{d}{=} \sP^*_0\mid_{\Psi^{(i)}((A^+_{T_{i}})^c)}.
$$
This equality in distribution also holds if we additionally condition on $\sP^*_0\mid_{\Psi^{(i)}(\fW)^c}$, since $\sP_0^*$ is a Poisson process, independent of all else, $\Psi^{(i)}(\fW)^c$ is again $\cG_{T_i}^+$-measurable, and $\Psi^{(i)}((A^+_{T_{i}})^c)$ and $\Psi^{(i)}(\fW)^c$ are disjoint. This implies that conditional on $\cG_{T_{i}}^+$, \eqref{E:Psi-II} is equal in distribution to $\sP^*_0\mid_{\Psi^{(i)}(A^+_{T_{i}})^c}$, giving \eqref{eq:sPi.dist.2}.
\end{proof}

\section{Proof of Proposition \ref{P:SI-to-BR}}
\label{S:SI-to-BR}

We use notation for SSPs from Section \ref{S:SSP} and notation for the colouring process in the SI model in Section \ref{S:SI-colouring}.
To check that we have defined an SSP, we need to check that all blue clocks take values in $[0, \ka(u, v)]$. This follows from condition (a), which guarantees that for any $u \sim v$ in $\Z^2$, that
$$
\tau_{f(v)} - \tau_{f(u)} \le \frac{\xi \ka(u, v)}{\ka_0}.
$$
The required acyclic condition on $0$-weighted edges follows since $X_\fR(u, v) \wedge X_\fB(u, v)$ can only equal $0$ if $u <_{\operatorname{SI}} v$. We need to check that this process has finite speed. 

Suppose that this is not the case. Then we can find an infinite chain $0 = u_0, u_1, \dots$ which is coloured in finite time, and each of the vertices $u_i$ in this chain is coloured by the edge from $u_{i-1}$. In particular, writing $T(u)$ for the time that a square $u$ is coloured in $(\fR, \fB)$, we have
$$
T(u_i) \ge X_\fR(u_{i-1}, u_i) \wedge X_\fB(u_{i-1}, u_i) + T(u_{i-1}) = X_\fR(u_{i-1}, u_i) + T(u_{i-1}),
$$
where the final inequality is by definition.
This implies that
$$
\sum_{i=1}^\infty X_\fR(u_{i-1}, u_i) < \infty.
$$
Since $\tau_{f(v)} - \tau_{f(u)} \le X_\fR(u, v)$ for all edges $(u, v)$, we also have that
$$
\limsup_{N \to \infty} \sum_{i=1}^N \tau_{f(u_i)} - \tau_{f(u_{i-1})} < \infty.
$$
The sum on the right is equal of $\tau_{f(u_N)}$. Therefore this implies that the infection hits infinitely many squares in finite time, contradicting the upper bound of linear growth from Theorem \ref{thm:KS.upper.bound}.

 Next, we check \eqref{E:BR-tau-eqn} and \eqref{E:fRinfty}. Because both the colouring process and the infection process proceed at finite speed, we can do this inductively. Let $t_0 = 0 < t_1 < t_2 < \dots \sset [0, \infty)$ denote the set of times in $[0, \infty)$ that are either equal to $T(u)$ or $\tau_{f(u)}$ for some $u \in \Z^2$. We show that for every $k = 0, 1, \dots,$ that for all $v$ with $T(v) \wedge \tau_{f(v)} \le t_k$, a.s.\ we have
 \begin{itemize}
 	\item $T(v) = \tau_{f(v)}$, and
 	\item if $f(v) \notin \operatorname{IGN}$ then $C(v) = \fB$ and $T(v) = T(u) + X_\fB(u, v)$ for some $u$ with $C(u) = \fB$.
 \end{itemize} 
 At time $t_0 = 0$, we have $T(0) = \tau_{f(0)} = 0$. Moreover, $\tau_B > 0$ for all $B \ne f(0)$ since no particles jump at time $0$ a.s. This also implies that no clocks in $(\fR, \fB)$ ring at time $0$, and hence $T(u) > 0$ for all $u \ne 0$, establishing the $k=0$ case of the inductive claim. 
 
 Now assume the inductive hypothesis holds up to time $t_k$. Consider time $t_{k+1}$, and let $v \in \Z^2$ be such that $t_{k+1} \in \{T(v), \tau_{f(v)} \}$. At least one of the following events holds:
 \begin{enumerate}
 	\item $T(v) = t_{k+1}$ for some $v \in \Z^2$, and there exists $u$ such that $T(u) < T(v)$ and the edge $(u, v)$ fired to colour $v$.
 	\item $\tau_{f(v)} = t_{k+1}$ for some $v \in \Z^2$. Moreover, for every vertex $w$ with $\tau_{f(v)} = \tau_{f(w)}$, we have $v \le_{\operatorname{SI}} w$. 
 	\item $T(v) = t_{k+1}$ for some $v \in \Z^2$, and there exists $u$ such that $T(u) = T(v)$ and the edge $(u, v)$ fired to colour $v$.
 	\item $\tau_{f(v)} = t_{k+1}$ for some $v \in \Z^2$ and there exists a vertex $w \in \Z^2$ with $\tau_{f(v)} = \tau_{f(w)}$ and $w <_{\operatorname{SI}} v$.
 \end{enumerate}

\textbf{Case 1: \qquad } 
  By the inductive hypothesis, we have $\tau_{f(v)} > t_k$ and $\tau_{f(u)} \le t_k$. In particular, $\tau_{f(v)} - \tau_{f(u)} > 0$. If the square $u$ were blue, then we immediately have $T(v) = \tau_{f(v)}$. 
  Now suppose the square $u$ is red. By the definition of $X_\fR$, since $\tau_{f(u)} < \tau_{f(v)}$, to show that $T(v) = \tau_{f(v)}$ we just need to show that
  \begin{equation}
  \label{E:kaxi}
  \tau_{f(v)} - \tau_{f(u)} < \frac{\xi}{\ka_0}.
  \end{equation}
  By the inductive hypothesis, the ignition particle for $f(u)$ is infected at time of ignition and the event $\cA_{f(u)}$ holds. Therefore an infected particle moves from $f(u)$ into the block $f(v)$ in the interval $[\tau_{f(u)}, \tau_{f(u)} +  \xi/\ka_0)$. This yields \eqref{E:kaxi}. 
  
  We now check the second bullet.
  If $f(v) \notin \operatorname{IGN}$, then there exists $w \sim v$ with $\tau_{f(w)} + \xi \ka(u, v)/\ka_0 = \tau_{f(v)}$, and such that $w$ is a coloured blue. By the inductive hypothesis, $\tau_{f(w)} = T(w)$, and so by the definition of $X_\fB$, we have
  $$
  T(v) \le T(w) + X_\fB(w, v) \le \tau_{f(w)} + (\tau_{f(v)} - \tau_{f(w)}) = \tau_{f(v)} = T(v).
  $$
  This would imply that the edge $(w, v)$ coloured the vertex $v$ blue, since blue colouring takes precedent over red colouring. Therefore $C(v) = \fB$ and $T(v) = T(w) + X_\fB(w, v)$, as desired.

\textbf{Case 2: \qquad }  We first claim that there exists a vertex $u \sim v$ with $\tau_{f(u)} < \tau_{f(v)}$. The existence of such a $u$ is clear if $v$ gets coloured by event (a) in Definition \ref{D:SI-colouring}, in which case $u$ satisfies $\tau_{f(u)} + \xi \ka(u, v)/\ka_0 = \tau_{f(v)}$. Moreover, since $v \le_{\operatorname{SI}} w$ for all vertices $w$ that get coloured at the same time as $v$, then $f(v)$ cannot get coloured by the presence of an infected particle in one of its neighbours. Therefore if $f(v)$ is not coloured by rule (a), then $f(v)$ gets coloured by rule (b). That is, an infected particle enters from an adjacent box $f(u)$ with $\tau_{f(u)} < \tau_{f(v)}$.

Now, we have $T(v) \ge t_{k+1}$ by the inductive hypothesis. On the other hand, the definitions of $X_\fR$ and $X_\fB$ imply that
$$
T(v) \le T(u) + \lf(\tau_{f(v)} - \tau_{f(u)} \rg) = \tau_{f(v)} = t_{k+1}.
$$
The first equality follows from the inductive hypothesis. Therefore $T(v) = \tau_{f(v)}$ as desired. 

We move to the second bullet. First, if $v \notin \operatorname{IGN}$, then $f(v)$ is coloured from an adjacent box $f(u)$ via event (a) in Definition \ref{D:SI-colouring}. Now, we have established that $T(v) =\tau_{f(v)}$, and also $T(u) =\tau_{f(u)}$ by the inductive hypothesis, so we have that
\begin{equation}
\label{E:TvTu}
T(v) = T(u) + \xi\ka(u, v)/\ka_0 = T(u) + X_\fB(u, v),
\end{equation}
where the second equality uses the definition of $X_\fB$. Note that from the definition of $X_\fR$, we also have that $T(v) > T(u) + X_\fR(u, v)$. Hence $C(u) = \fB$, since otherwise $v$ would have been coloured earlier in $(\fR, \fB)$.
Finally, since blue colouring takes precedent over red colouring, \eqref{E:TvTu} implies that $C(v) = \fB$ and $v$ was coloured by the clock $X_\fB(u, v)$.

\textbf{Case 3: \qquad } We can find a chain of vertices $u_0 \sim u_1 \sim \dots u_{k-1} = u \sim u_k = v$ for some $k \ge 2$ such that the edge $(u_{i-1}, u_i)$ colours the vertex $u_i$ for all $i$, $T(u_1) = T(v)$ and $T(u_0) < T(u_1)$. By Case 1, $\tau_{f(u_1)} = T(u_1)$. The definitions of $X_R(u_{i-1}, u_i)$ and $X_B(u_{i-1}, u_i)$ then ensure that 
$
\tau_{f(u_{i-1})} = \tau_{f(u_i)},
$
for all $i$, and hence $\tau_{f(v)} = T(v)$. Moreover, since $\tau_{f(v)} = \tau_{f(u)}$ and $u$ and $v$ are comparable in the order $<_{\operatorname{SI}}$, the box $f(v)$ was coloured in $X_t$ according to rule (b) or (c), so $f(v)$ was necessarily ignited.

\textbf{Case 4: \qquad } Let $G$ denote the random graph with vertex set $\Z^2$ and edges $(u_1, u_2)$ whenever $u_1 <_{\operatorname{SI}} u_2$, and let $H$ be the component of this graph containing $w, v$. Note that $H$ has cardinality either $2$ or $3$. For notational simplicity, we proceed in the case when $H$ has cardinality $3$; the other case follows the same argument. Let $v_1, v_2, v_3$ be the three elements of $H$, listed in the $<_{\operatorname{SI}}$ order. We have $\tau_{f(v_i)} = t_{k+1}$ for all $i$, and so the inductive hypothesis ensures that
\begin{equation}
\label{E:taufvii}
\tau_{f(v_i)} \le T(v_i)
\end{equation}
for all $i$.
The vertex $v_1$ satisfies the assumptions of Case $2$. Therefore $T(v_1) = \tau_{f(v_1)}$. The colouring rules then imply that 
$$
T(v_3) \le T(v_2) \le T(v_1),
$$
which implies that \eqref{E:taufvii} is in fact an equality. Finally, since each of the $f(v_i)$ were infected by either event (b) or (c), we have $v_i \in \operatorname{IGN}$ for all $i$.

 \bibliographystyle{alpha}
 
\bibliography{SIR}

\end{document}